\documentclass[12pt, reqno]{amsart}
\usepackage{}
\usepackage{cancel}
\usepackage{stmaryrd}
\usepackage{amsmath}
\usepackage{amsfonts}
\usepackage{amssymb}
\usepackage[all,cmtip]{xy}           
\usepackage{xspace}
\usepackage{bbding}
\usepackage{makecell}
\usepackage{txfonts}
\usepackage[shortlabels]{enumitem}
\usepackage{cite}
\usepackage{xcolor} 

\usepackage{tikz}
\usepackage{titletoc}

\usepackage{ifpdf}
\ifpdf
 \usepackage[colorlinks,final,backref=page,hyperindex]{hyperref}
\else
 \usepackage[colorlinks,final,backref=page,hyperindex,hypertex]{hyperref}
\fi
\usepackage[active]{srcltx}

\usepackage{comment}	

\makeatletter

\begin{document}
\newcommand {\emptycomment}[1]{} 
\newcommand{\bluefix}[1]{{\color{blue}#1}} 

\baselineskip=15pt
\newcommand{\nc}{\newcommand}
\newcommand{\delete}[1]{}

\nc{\mlabel}[1]{\label{#1}}  
\nc{\mcite}[1]{\cite{#1}}  
\nc{\mref}[1]{\ref{#1}}  
\nc{\meqref}[1]{\eqref{#1}} 
\nc{\mbibitem}[1]{\bibitem{#1}} 

\delete{
\nc{\mlabel}[1]{\label{#1}  
{\hfill \hspace{1cm}{\bf{{\ }\hfill(#1)}}}}
\nc{\mcite}[1]{\cite{#1}{{\bf{{\ }(#1)}}}}  
\nc{\mref}[1]{\ref{#1}{{\bf{{\ }(#1)}}}}  
\nc{\meqref}[1]{\eqref{#1}{{\bf{{\ }(#1)}}}} 
\nc{\mbibitem}[1]{\bibitem[\bf #1]{#1}} 
}

\nc{\mrm}[1]{{\rm #1}}
\nc{\id}{\mrm{id}}  \nc{\Id}{\mrm{Id}}
\nc{\admset}{\{\pm x\}\cup (-x+K^{\times}) \cup K^{\times} x^{-1}}

\nc{\vsa}{\vspace{-.1cm}}
\nc{\vsb}{\vspace{-.2cm}}
\nc{\vsc}{\vspace{-.3cm}}
\nc{\vsd}{\vspace{-.4cm}}
\nc{\vse}{\vspace{-.5cm}}

\def\a{\alpha}
\def\aa{\mathfrak{A}}
\def\admt{admissible to~}
\def\ad{associative D-}
\def\asi{ASI~}
\def\aybe{aYBe~}
\def\b{\beta}
\def\bd{\boxdot}
\def\bbf{\overline{f}}
\def\bF{\overline{F}}
\def\bbF{\overline{\overline{F}}}
\def\bbbf{\overline{\overline{f}}}
\def\bg{\overline{g}}
\def\bG{\overline{G}}
\def\bbG{\overline{\overline{G}}}
\def\bbg{\overline{\overline{g}}}
\def\bT{\overline{T}}
\def\bt{\overline{t}}
\def\bbT{\overline{\overline{T}}}
\def\bbt{\overline{\overline{t}}}
\def\bR{\overline{R}}
\def\br{\overline{r}}
\def\bbR{\overline{\overline{R}}}
\def\bbr{\overline{\overline{r}}}
\def\bu{\overline{u}}
\def\bU{\overline{U}}
\def\bbU{\overline{\overline{U}}}
\def\bbu{\overline{\overline{u}}}
\def\bw{\overline{w}}
\def\bW{\overline{W}}
\def\bbW{\overline{\overline{W}}}
\def\bbw{\overline{\overline{w}}}
\def\btl{\blacktriangleright}
\def\btr{\blacktriangleleft}
\def\calo{\mathcal{O}}
\def\ci{\circ}
\def\d{\delta}
\def\dd{\diamondsuit}
\def\D{\Delta}
\newcommand{\End}{\mathrm{End}}
\def\frakB{\mathfrak{B}}
\def\G{\kappa}
\def\g{\gamma}
\def\gg{\mathfrak{g}}
\def\hh{\mathfrak{h}}
\def\k{\kappa}
\def\l{\lambda}
\def\ll{\mathfrak{L}}
\def\lll{\mathcal{L}}
\def\rrr{\mathcal{R}}
\def\lh{\leftharpoonup}
\def\lr{\longrightarrow}
\def\N{Nijenhuis~}
\def\o{\otimes}
\def\om{\omega}
\def\opa{\cdot_{A}}
\def\opb{\cdot_{B}}
\def\p{\psi}
\def\sadm{$S$-admissible~}
\def\r{\rho}
\def\ra{\rightarrow}
\def\rep{representation~}
\def\rh{\rightharpoonup}
\def\rr{r^{\#}}
\def\s{\sigma}
\def\st{\star}
\def\ss{\mathfrak{sl}(2)}
\def\sss{$\mathbb{S}$}
\def\ti{\times}
\def\tl{\triangleright}
\def\tr{\triangleleft}
\def\v{\varepsilon}
\def\vp{\varphi}
\def\vth{\vartheta}

\newtheorem{thm}{Theorem}[section]
\newtheorem{lem}[thm]{Lemma}
\newtheorem{cor}[thm]{Corollary}
\newtheorem{pro}[thm]{Proposition}
\theoremstyle{definition}
\newtheorem{defi}[thm]{Definition}
\newtheorem{ex}[thm]{Example}
\newtheorem{rmk}[thm]{Remark}
\newtheorem{pdef}[thm]{Proposition-Definition}
\newtheorem{condition}[thm]{Condition}
\newtheorem{question}[thm]{Question}
\renewcommand{\labelenumi}{{\rm(\alph{enumi})}}
\renewcommand{\theenumi}{\alph{enumi}}

\nc{\li}[1]{\textcolor{red}{Li:#1}}

\font\cyr=wncyr10


\title{Nijenhuis pre-Lie bialgebras, Nijenhuis Lie bialgebras and \\ the $\mathbb{S}$-equation}

\author[Guo]{Li Guo}
\address{Department of Mathematics and Computer Science, Rutgers University, Newark, NJ 07102, USA}
         \email{liguo@rutgers.edu}

\author[Ma]{Tianshui Ma}
\address{School of Mathematics and Statistics, Henan Normal University, Xinxiang 453007, China}
\email{matianshui@htu.edu.cn}


\begin{abstract}
We develop a bialgebraic framework for Nijenhuis operators on pre-Lie algebras and coalgebras. Starting from dual triangular pre-Lie bialgebras and weak pseudo-Hessian structures, we construct Nijenhuis operators on pre-Lie algebras and establish the corresponding coalgebraic construction. We introduce Nijenhuis pre-Lie bialgebras, characterize them through matched pairs, and identify compatibility conditions linking the two constructions. For triangular structures, we formulate the $S$-Nijenhuis $\mathbb{S}$-equation, give an operator characterization, and obtain constructions using relative Rota-Baxter operators. We also show that balanced Nijenhuis pre-Lie bialgebras give rise to Nijenhuis Lie bialgebras. Finally, we introduce crossed modules and braided Nijenhuis pre-Lie bialgebras, and characterize when compatible algebraic and coalgebraic data assemble into a Nijenhuis pre-Lie biproduct. These results connect Nijenhuis operators with pre-Lie bialgebra structures, Yang--Baxter-type equations, and constructions on representations.
\end{abstract}

\subjclass[2020]{
17B38,  
16T25,   
16T10,   
17B62,   
17A30.  
}

\keywords{Nijenhuis operators; pre-Lie bialgebras; weak pseudo-Hessian pre-Lie algebras, Lie bialgebras}

\maketitle

\vspace{-1.04cm}

 \tableofcontents

\vspace{-1.4cm}

\newpage

\allowdisplaybreaks

\section{Introduction}\label{se:ip}
Nijenhuis operators originate in differential geometry \cite{Nij}, where they arise in the study of geometric structures and their compatible deformations. Their algebraic counterparts now play a significant role in deformation theory and the construction of integrable structures \cite{BKM22,BP,CJGM,GZ,LMMP7,KM,LG,SX1,SX2}. A recurring theme is that a Nijenhuis operator does more than preserve an existing operation: it can produce a new operation related to the original one by a controlled deformation. This makes it natural to study such operators in algebraic settings that retain more structure than a Lie bracket alone. For further studies of Nijenhuis operators in the pre-Lie and Lie settings, see \cite{WSBL,LM}.

A (left) {\bf pre-Lie algebra} is a vector space $A$ with a multiplication $\ci$ such that
$$(x\ci y)\ci z-x\ci (y\ci z)=(y\ci x)\ci z-y\ci (x\ci z), \forall x, y, z\in A.$$
The term pre-Lie algebra comes from the fact that the map $(x,y)\mapsto x\ci y-y\ci x$
defines a Lie bracket. Pre-Lie algebras, also known as Vinberg or left-symmetric algebras, arise in geometry, deformation theory, and perturbative renormalization \cite{V,Ger,CK}; see \cite{Bai3,Man} for surveys. For further work on the structure and classification of left-symmetric algebras, see \cite{BM,Bur}.

The passage from a pre-Lie product to its commutator Lie bracket is particularly relevant here. A construction on a pre-Lie algebra can induce a construction on the associated Lie algebra, but the converse need not determine the original pre-Lie product. Thus, to understand Nijenhuis phenomena at the pre-Lie level, it is important to keep track of the product itself and of the left and right actions appearing in its representations. This additional information becomes essential when algebraic and coalgebraic operations are considered together.

Pre-Lie bialgebras were introduced in connection with para-K\"ahler Lie algebras and phase spaces of Lie algebras \cite{Bai2,BB}. Their triangular structures are governed by the $\mathbb{S}$-equation, a pre-Lie analogue of the classical Yang--Baxter equation \cite{Dr,HB}. Symmetric tensors and bilinear forms therefore offer two complementary ways to encode compatibility: the former can specify a coproduct, while the latter can be used to construct or analyze a pre-Lie multiplication. Related variants include Hom-pre-Lie bialgebras \cite{LMS}. Further developments include quasi-triangular pre-Lie bialgebras \cite{WBLS}, left-symmetric conformal bialgebras \cite{HL}, and pre-Lie $2$-bialgebras \cite{LYW}, alongside the related theory of Lie $2$-bialgebras \cite{BSZ}. Additional geometric and duality-oriented perspectives can be found in studies of para-K\"ahler structures \cite{MO} and double constructions of Frobenius algebras \cite{Bai}. The dual triangular and weak pseudo-Hessian structures developed below make this interplay explicit.

In a bialgebraic setting, however, the existence of a Nijenhuis operator on the multiplication side does not by itself provide a compatible operator on the comultiplication side. One must also relate the two operators to the pre-Lie bialgebra compatibility conditions. This distinction motivates a framework with an algebraic operator $N$ and a coalgebraic operator $S$, together with their respective admissibility conditions. It also explains why matched pairs, dual representations, and corepresentations are natural tools: they express the interaction of the two sides without reducing it to a single operation.

The aim of this paper is to develop such a framework and to provide concrete constructions. We first use dual triangular pre-Lie bialgebras and weak pseudo-Hessian structures to construct Nijenhuis operators on pre-Lie algebras, and then establish the corresponding coalgebraic construction. We define Nijenhuis pre-Lie bialgebras and characterize them in terms of matched pairs. We further obtain an $S$-Nijenhuis $\mathbb{S}$-equation and constructions involving relative Rota--Baxter operators. For related Rota--Baxter bialgebra constructions in Lie and associative algebra settings, see \cite{BGLM,BGM}. The balanced case provides a connection with Nijenhuis Lie bialgebras.

A further structural issue is how to enlarge a Nijenhuis pre-Lie bialgebra using a space carrying both action and coaction data. In the theory of Lie bialgebras, braided and biproduct constructions provide a model for organizing such compatibility \cite{Maj}. The pre-Lie setting requires its own formulation because both left and right actions occur, and the Nijenhuis operators impose additional conditions. We therefore introduce crossed modules and braided Nijenhuis pre-Lie bialgebras, culminating in a criterion for constructing biproduct bialgebras on direct sums. This connects the foundational operator theory to a method of producing new examples.

The paper is organized as follows. Section \ref{se:preliebialg} develops dual triangular and weak pseudo-Hessian structures, and Section \ref{se:suff} applies them to Nijenhuis operators. Section \ref{se:rep} establishes the representation-theoretic and coalgebraic framework, matched-pair characterizations, the $S$-Nijenhuis $\mathbb{S}$-equation, and relative Rota--Baxter constructions. Section \ref{se:balance} relates the resulting structures to Nijenhuis Lie bialgebras. Section \ref{se:bpre} studies crossed modules, braided structures, and the biproduct construction.

\smallskip
\noindent{\bf Notations:} Throughout this paper, we fix a field $K$. All vector spaces, tensor products, and linear homomorphisms are taken over $K$. We denote by $\id_M$ the identity operator on a vector space $M$, and by $\sigma$ the flip map from $M\o M$ to $M\o M$. Unless explicitly stated otherwise, all vector spaces are finite-dimensional over $K$.

\section{Dual triangular pre-Lie bialgebras and weak pseudo-Hessian  pre-Lie algebras}\label{se:preliebialg}
We begin with the bialgebraic structures that provide the ingredients for our operator constructions. Dual triangular pre-Lie bialgebras naturally lead to weak pseudo-Hessian pre-Lie algebras.

\subsection{Dual triangular pre-Lie bialgebras}\label{se:tri}
We first recall pre-Lie coalgebras \cite{LZB,MLRBC} and pre-Lie bialgebras \cite{Bai2,LZB} before turning to their dual triangular forms.

 \begin{defi}
\begin{enumerate}
\item A {\bf pre-Lie coalgebra} is a pair $(A, \D)$ consisting of a vector space $ A$ and a linear map $\D:  A\lr  A\o  A$ (we use the sumless Sweedler notation \cite{Sw} and write $\D(x)=x_{(1)}\o x_{(2)}(:=\sum x_{(1)}\o x_{(2)})$) such that for all $x\in A$,
 \begin{eqnarray}\label{eq:cl}
 x_{(1)(1)}\o x_{(1)(2)}\o x_{(2)}-x_{(1)}\o x_{(2)(1)}\o x_{(2)(2)}=x_{(1)(2)}\o x_{(1)(1)}\o x_{(2)}-x_{(2)(1)}\o x_{(1)}\o x_{(2)(2)}.
 \end{eqnarray}
\item A {\bf pre-Lie bialgebra} is a triple $(A, \ci, \D)$ consisting of a pre-Lie algebra $(A, \ci)$ and a pre-Lie coalgebra $(A, \D)$ such that, for all $x, y\in  A$, the following conditions hold:
 \begin{eqnarray}
 &\D([x,y]_{\ci})=x\ci y_{(1)}\o y_{(2)}+y_{(1)}\o [x,y_{(2)}]_{\ci}-y\ci x_{(1)}\o x_{(2)}-x_{(1)}\o [y,x_{(2)}]_{\ci},&\label{eq:b1}\\
 &(\id\o \id-\sigma)\D(x\ci y)=(\id\o \id-\sigma)(x_{(1)}\o x_{(2)}\ci y+y_{(1)}\o x\ci y_{(2)}-y_{(2)}\o x\ci y_{(1)}),&\label{eq:b2}
 \end{eqnarray}
 where $[x, y]_{\ci}:=x\ci y-y\ci x$.
\end{enumerate}
\label{de:lb}
\end{defi}

The following is the pre-Lie analog of the construction of Lie bialgebras from $r$-matrices \cite{Dr}.

\begin{pro} \label{pro:qt} {\em $($\cite[Proposition 6.1]{Bai2}$)$} Let $(A, \ci)$ be a pre-Lie algebra. For $r=\sum\limits^n_{i=1} a_i\o b_i\in  A\o A$, define
 \begin{equation}\label{eq:cop}
 \D_r(x)=\sum^n_{i=1} (x\ci a_i\o b_i+a_i\o x\ci b_i-a_i\o b_i\ci x).
 \end{equation}
 If $r\in  A\o A$ is a symmetric solution of the following {\bf \sss-equation} in $(A, \ci)$:
 \begin{equation}\label{eq:cybe}
 \sum^n_{i,j=1} (a_i\o b_i\ci a_j\o b_j+a_i\o a_j\o b_i\ci b_j)
 =\sum^n_{i,j=1} (a_i\ci a_j\o b_i\o b_j+a_i\o a_j\o b_j\ci b_i),
 \end{equation}
 then $(A, \ci, \D_r)$ is a pre-Lie bialgebra. We call this pre-Lie bialgebra {\bf triangular} and denote it by $(A, \ci, r, \D_r)$.
 \end{pro}

 Now we give a characterization of triangular pre-Lie bialgebras.

 \begin{pro}\label{pro:qt-1} Let $(A, \ci)$ be a pre-Lie algebra, $r\in  A\o A$ and $\D_r$ be defined by Eq.~\meqref{eq:cop}. Then Eq.\,(\ref{eq:cybe}) holds if and only if
 \begin{eqnarray}\label{eq:qt2}
 (\id\o \D_r)(r)=\sum^n_{i,j=1} a_i\ci a_j\o b_i\o b_j.
 \end{eqnarray}
 Further if $r$ is symmetric, then the quadruple $(A, \ci, r, \D_r)$ is a triangular pre-Lie bialgebra if and only if
 \begin{eqnarray}\label{eq:qt2-1}
 (\D_r\o \id)(r)=\sum^n_{i,j=1} a_i\o a_j\o b_i\ci b_j.
 \end{eqnarray}
 \end{pro}

 \begin{proof} It is a direct computation by Eqs.~(\mref{eq:cop}) and (\mref{eq:cybe}).
 \end{proof}

To state the first main result, we need another notion.

\begin{defi}\label{de:ccybe} Let $(A, \D)$ be a pre-Lie coalgebra and $\om\in (A\o A)^*$ a symmetric bilinear form (in the sense of $\om(x, y)=\om(y, x)$). For all $x, y, z\in  A$, the equation
 \begin{eqnarray}\label{eq:ccybe}
 \om(x, y_{(1)})\om(y_{(2)}, z)+\om(x, z_{(1)})\om(y, z_{(2)})=\om(x_{(1)}, y)\om(x_{(2)}, z)+\om(y, z_{(1)})\om(x, z_{(2)})
 \end{eqnarray}
 is called a {\bf co-\sss-equation in $(A, \D)$}.
 \end{defi}

 \begin{thm}\label{thm:cqt} Let $(A, \D)$ be a pre-Lie coalgebra and $\om\in (A\o A)^*$ be a symmetric solution of the co-\sss-equation in $(A, \D)$. Define a multiplication on $ A$ by
 \begin{eqnarray}\label{eq:p}
 x\ci_{\om} y:=x_{(1)}\om(x_{(2)}, y)+y_{(1)}\om(x, y_{(2)})-\om(x, y_{(1)})y_{(2)},\quad \forall~x, y\in  A.
 \end{eqnarray}
 Then $(A, \ci_{\om}, \D)$ is a pre-Lie bialgebra. We call this bialgebra a {\bf dual triangular pre-Lie bialgebra} and denote it by $(A, \D, \om, \ci_{\om})$.
 \end{thm}

\begin{proof}
The bialgebra compatibility identities \eqref{eq:b1} and \eqref{eq:b2} follow by substituting \eqref{eq:p} and using the symmetry of $\om$. It remains to verify the pre-Lie identity. Set $\mathcal A_\om(x,y,z)=(x\ci_\om y)\ci_\om z-x\ci_\om(y\ci_\om z)$.
Expanding the four terms of $\mathcal A_\om(x,y,z)-\mathcal A_\om(y,x,z)$ by \eqref{eq:p}, the terms containing iterated coproducts cancel in pairs by the pre-Lie coalgebra identity \eqref{eq:cl}. The remaining terms are contractions of the two sides of the co-$\mathbb{S}$-equation \eqref{eq:ccybe} against the coproduct components of $x$, $y$, and $z$. They therefore vanish. Thus $\mathcal A_\om(x,y,z)=\mathcal A_\om(y,x,z)$, so $(A,\ci_\om)$ is a pre-Lie algebra, as required.
\end{proof}

 The following proposition gives a useful criterion for the dual triangular construction.

 \begin{pro}\label{pro:cqt} Let $(A, \D)$ be a pre-Lie coalgebra, $\om\in (A\o A)^*$ and $\ci_{\om}$ be defined in Eq.~\meqref{eq:p}. Then Eq.\,(\mref{eq:ccybe}) holds if and only if
 \begin{eqnarray}\label{eq:cqt2}
 \om(x, y\ci_\om z)=\om(x_{(1)}, y)\om(x_{(2)}, z), \quad x, y, z\in A.
 \end{eqnarray}
 If further, $\om$ is a symmetric bilinear form, then the quadruple $(A, \D, \om, \ci_{\om})$ is a dual triangular pre-Lie bialgebra if and only if
 \begin{eqnarray}\label{eq:cqt2-1}
 \om(x\ci_\om y, z)=\om(x, z_{(1)})\om(y, z_{(2)}), \quad x, y, z\in A.
 \end{eqnarray}
 \end{pro}

 \begin{proof} The proof is obtained by taking the dual to Proposition \ref{pro:qt-1}, following the proof of Theorem \ref{thm:cqt}.
 \end{proof}

\subsection{Weak pseudo-Hessian pre-Lie algebras}\label{se:sym}
The dual triangular construction above motivates the following bilinear-form condition, which will be used to produce Nijenhuis operators.

 \begin{defi}\label{de:sym} Let $(A, \ci)$ be a pre-Lie algebra and $\om\in (A\o  A)^*$ a symmetric bilinear form. Assume that for all $x, y, z\in  A$,
 \begin{eqnarray}\label{eq:symp}
 \om(x\ci y, z)-\om(x, y\ci z)=\om(y\ci x, z)-\om(y, x\ci z).
 \end{eqnarray}
 Then we call $(A, \ci)$ a {\bf weak pseudo-Hessian pre-Lie algebra} and denote it by $(A, \ci, \om)$.
 \end{defi}

 \begin{rmk}
 \begin{enumerate}
   \item If $\om\in (A\o  A)^*$ is nondegenerate, then a weak pseudo-Hessian pre-Lie algebra is a pseudo-Hessian pre-Lie algebra introduced by Bai and Ni (\cite{NB}).
   \item The bilinear form in Definition \ref{de:sym} is precisely the definition of a 2-cocycle of a pre-Lie algebra into the trivial bimodule $K$ (see \cite{NB,SW}). It is equivalent to the following central extension: there exists a pre-Lie algebra structure on $A\oplus Kc$ given by
 $$x\star y=x\ci y+\om(x,y)c,~x\star c=c\star x=c\star c=0, \quad x, y\in A,$$
 if and only if $\om$ is a 2-cocycle of $A$.
 \end{enumerate}
 \end{rmk}

A dual triangular pre-Lie bialgebra induces a weak pseudo-Hessian  pre-Lie algebra.

 \begin{pro}\label{pro:de:sym} Let $(A, \D)$ be a pre-Lie coalgebra and $\om\in (A\o A)^*$ a symmetric bilinear form. If $(A, \D, \om, \ci_{\om})$ is a dual triangular pre-Lie bialgebra, where $\ci_{\om}$ is defined by Eq.~\meqref{eq:p}, then $(A, \ci_\om, \om)$ is a weak pseudo-Hessian  pre-Lie algebra.
 \end{pro}

 \begin{proof} For all $x, y, z\in  A$, we have
 \begin{eqnarray*}
 &&\hspace{-6mm}\om(x\ci_\om y, z)-\om(x, y\ci_\om z) \\
 &&\stackrel{}{=}\om(x_{(1)}, z)\om(x_{(2)}, y)+\om(y_{(1)}, z)\om(x, y_{(2)})-\om(x, y_{(1)})\om(y_{(2)}, z)-\om(x_{(1)}, y)\om(x_{(2)}, z)\\
 &&\hspace{10mm}(\hbox{~by Eqs.(\mref{eq:p}) and (\mref{eq:cqt2})})\\
 &&\stackrel{}{=}\om(y_{(1)}, z)\om(y_{(2)}, x)+\om(x_{(1)}, z)\om(y, x_{(2)})-\om(y, x_{(1)})\om(x_{(2)}, z)-\om(y_{(1)}, x)\om(y_{(2)}, z)\\
 &&\hspace{10mm} \hbox{~(by the symmetry of } \om)\\
 &&\stackrel{}{=}\om(y\ci_\om x, z)-\om(y, x\ci_\om z), \hspace{3mm}(\hbox{~by Eqs.(\mref{eq:p}) and (\mref{eq:cqt2})}),
 \end{eqnarray*}
 finishing the proof.
 \end{proof}

By a direct computation, we have the following example.

 \begin{ex} \label{ex:sym} Let $ A=K\{e, f\}$ be a 2-dimensional vector space and $\ell_i, i=1, 2, 3, \l$ be parameters. By \cite{BM} or \cite{Bur}, $(A, \ci)$ is a pre-Lie algebra, where $\ci$ is given by
 \begin{center}
 \begin{tabular}{r|cc}
          $\ci$ & $e$  & $f$  \\
          \hline
           $e$ & $0$  & $0$  \\
           $f$ & $-e$  &  $\l f$ \\
        \end{tabular} \quad$(\l\neq -1)$.
 \end{center}
 Define $\om$ by
 \begin{center}
          (a)\quad \begin{tabular}{r|cc}
            $\om$ & $e$  & $f$  \\
            \hline
            $e$ & $0$  & $\ell_3$  \\
            $f$ & $\ell_3$  &  $\ell_2$ \\
          \end{tabular}\quad $(\l=1)$
         \qquad \quad~or \qquad (b)\quad \begin{tabular}{r|cc}
            $\om$ & $e$  & $f$  \\
            \hline
            $e$ & $0$  & $0$  \\
            $f$ & $0$  &  $\ell_2$ \\
          \end{tabular}
          \end{center}
Then the triple $(A, \ci, \om)$ is a weak pseudo-Hessian pre-Lie algebra.
 \end{ex}

\section{Nijenhuis pre-Lie algebras from pre-Lie bialgebras}\label{se:suff}
The weak pseudo-Hessian structures of Section \ref{se:preliebialg} now yield Nijenhuis operators under suitable triangular bialgebra conditions. We first recall the definition from \cite{WSBL}.

 \begin{defi}\label{de:npre} Let $(A, \ci)$ be a pre-Lie algebra and $N: A\lr  A$ a linear map. If
 \begin{eqnarray}
 &N(x)\ci N(y)+N^2(x\ci y)=N(N(x)\ci y)+N(x\ci N(y)),\quad x, y\in A,&\label{eq:n}
 \end{eqnarray}
 then we call $N$ a {\bf Nijenhuis operator on $(A, \ci)$} and the pair $((A, \ci), N)$ a {\bf Nijenhuis pre-Lie algebra}.
 \end{defi}

 \begin{ex}\label{ex:de:npre} Let $(A, \ci)$ be a pre-Lie algebra.
 \begin{enumerate}[(1)]
   \item $N=\id_A$ is a \N operator on $(A, \ci)$;
   \item Define the left and right multiplication maps $L, R:  A\lr \End(A)$ by $L_x y=x\ci y$ and $R_x y=y\ci x$. If a linear map $N: A\lr A$ commutes with the left (or right) multiplication, i.e., for all $x\in A$, $N L_x=L_x N$ (or $N R_x=R_x N$), then $N$ is a \N operator.
 \end{enumerate}
 \end{ex}

 Now we present the main result in this section.

 \begin{thm}\label{thm:ln} Let $(A, \ci, \om)$ be a weak pseudo-Hessian pre-Lie algebra and $r=\sum\limits^n_{i=1} a_i\o b_i\in A\o  A$. Define a linear map
 \begin{equation}
 N: A\lr  A, \quad N(x)=\sum\limits_{i=1}\limits^n \om(x, a_i)b_i, \quad x\in A.\label{eq:ys}
 \end{equation}
 If $(A, \ci, r, \D_r)$ is a triangular pre-Lie bialgebra with the comultiplication $\D_r$ defined in Eq.~\meqref{eq:cop}, and $(A, \D_r, \om, \ci_{\om})$ is a dual triangular pre-Lie bialgebra with the multiplication $\ci_{\om}$ defined in Eq.~\meqref{eq:p}, then $((A, \ci), N)$ is a \N pre-Lie algebra.
 \end{thm}

\begin{proof} We first prove a preliminary equality. For all $x, y, z\in  A$, by Eq.\,(\mref{eq:ccybe}), one has
 \begin{eqnarray*}
 \om(x, y_{(1)})\om(y_{(2)}, z)+\om(x, z_{(1)})\om(y, z_{(2)})-\om(x_{(1)}, y)\om(x_{(2)}, z)-\om(y, z_{(1)})\om(x, z_{(2)})=0.
 \end{eqnarray*}
 Since $(A, \D_r, \om, \ci_{\om})$ is a dual triangular pre-Lie bialgebra, by Eq.\,(\mref{eq:cop}), we have
\begin{eqnarray*}
 &&\sum\limits_{i=1}\limits^n \Big(\om(x, y a_i)\om(b_i, z)+\om(x, a_i)\om(y b_i, z)-\om(x, a_i)\om(b_i y, z)\\
 &&+\om(x, z a_i)\om(y, b_i)+\om(x, a_i)\om(y, z b_i)-\om(x, a_i)\om(y, b_i z)\\
 &&-\om(x a_i, y)\om(b_i, z)-\om(a_i, y)\om(x b_i, z)+\om(a_i, y)\om(b_i x, z)\\
 &&-\om(y, z a_i)\om(x, b_i)-\om(y, a_i)\om(x, z b_i)+\om(y, a_i)\om(x, b_i z)\Big)=0.
 \end{eqnarray*}

 Then by Eq.\,(\mref{eq:symp}) and the symmetry of $r$, we obtain
 \begin{eqnarray}
 &\sum\limits_{i=1}\limits^n \Big(\om(a_i, z)\om(x y, b_i)-\om(x, a_i)\om(b_i, y z)-\om(a_i, z)\om(y x, b_i)+\om(a_i, y)\om(b_i, x z)\Big)=0.&\label{eq:cqt-sym}
 \end{eqnarray}

We next verify that the linear map $N$ defined in Eq.~\meqref{eq:ys} is a Nijenhuis operator on $(A, \ci)$.
\begin{eqnarray*}
 &&\hspace{-9mm}N(x)\ci N(y)-N(N(x)\ci y+x\ci N(y)-N(x\ci y))\\
 &\stackrel{ }{=}&\sum\limits_{i,j=1}\limits^n \Big(\om(x, a_i)\om(y, a_j)b_i\ci b_j
 -\om(x, a_i)\om(b_i\ci y, a_j)b_j-\om(y, a_i)\om(x\ci b_i, a_j)b_j\\
 &&+\om(x\ci y, a_i)\om(b_i, a_j)b_j\Big)\hspace{6mm}(\hbox{by Eq.\,(\mref{eq:ys})})\\
 &\stackrel{}{=}&\sum\limits_{i,j=1}\limits^n \Big(\om(x, a_i)\om(y, a_j)b_i\ci b_j
 -\om(x, a_i)\om(b_i\ci y, a_j)b_j-\om(y, a_i)\om(x\ci b_i, a_j)b_j\\
 &&+\om(x, a_i)\om(a_j\ci y, b_i)b_j+\om(y, a_i)\om(a_j\ci x, b_i)b_j-\om(y, a_i)\om(x\ci a_j, b_i)b_j\Big) \hspace{3mm}(\text{by Eq.}\,(\mref{eq:cqt-sym}))\\
 &\stackrel{}{=}&\sum\limits_{i,j=1}\limits^n \Big(\om(x, a_i)\om(y, a_j)b_i\ci b_j+\om(x, a_i)\om(y, b_i\ci a_j)b_j-\om(x, a_i)\om(y, a_j\ci b_i)b_j\\
 &&-\om(y, a_i)\om(x, a_j\ci b_i)b_j\Big) \hspace{6mm}(\hbox{by Eq.\,(\mref{eq:symp}) and the symmetry of } \om)\\
 &\stackrel{}{=}&0.\hspace{6mm}(\hbox{by Eq.\,(\mref{eq:cybe}) and the symmetry of } r)
 \end{eqnarray*}

 Therefore, $(A, \ci, N)$ is a \N pre-Lie algebra.
 \end{proof}

 Let $V$ be a vector space and $r\in V\o V$. Define $\rr: V^*\lr V$ by
 \begin{eqnarray}\label{eq:rr}
 \rr(v^*):=\sum\limits_{i=1}\limits^n \langle v^*, a_i \rangle b_i, ~~\forall~v^*\in V^*.
 \end{eqnarray}
 The element $r$ is called {\bf nondegenerate} if $\rr$ is bijective.
 \smallskip

 We next give some examples, beginning with the case when $r$ is nondegenerate.

 \begin{ex} Let $(A, \ci, r, \D_r)$ be a triangular pre-Lie bialgebra with the comultiplication $\D_r$ defined in Eq.~\meqref{eq:cop}. By \cite[Theorem 6.3]{Bai2}, when $r$ is nondegenerate, then $r$ is a symmetric solution of the \sss-equation in $(A, \ci)$ if and only if $\om_r$ is a pseudo-Hessian structure on $(A, \ci)$, where $\om_r$ is defined by
 \begin{eqnarray}\label{eq:inr}
 \om_r(x, y):=\langle {r^\#}^{-1}(x), y \rangle.
 \end{eqnarray}
 We note that $\om_r$ is symmetric if and only if $r$ is symmetric. Further, $(A, \D_r, \om_r, \ci_{\om_r})$ is a dual triangular pre-Lie bialgebra with the multiplication $\ci_{\om_r}$ defined in Eq.~\meqref{eq:p}.
 In this case, the Nijenhuis operator $N$ given by Eq.\,(\ref{eq:ys}) on $(A, \ci)$ is
 \begin{eqnarray*}
 &N(x)=\sum\limits_{i=1}\limits^n \om_r(x, a_i)b_i\stackrel{(\mref{eq:inr})}{=}\sum\limits_{i=1}\limits^n \langle {r^\#}^{-1}(x), a_i\rangle b_i\stackrel{(\mref{eq:rr})}{=}r^\#({r^\#}^{-1}(x))=x,~~\forall~x\in  A.&
 \end{eqnarray*}
 That is, $N=\id_A$.
 \label{ex:trivial}
\end{ex}

\begin{ex} \label{ex:thm:ln} Let $(A, \ci, \om)$ be a weak pseudo-Hessian pre-Lie algebra given in Example \mref{ex:sym}. All the Nijenhuis operators on $(A, \ci)$ by using the method of Theorem \mref{thm:ln} are given as follows. Here $k_i, \ell_i, i=1, 2, 3$ are parameters.
 \begin{enumerate}[(a)]
   \item For case (a) in Example~\mref{ex:sym}:
   $$N=0,\quad \hbox{or} \quad  \left\{
            \begin{array}{l}
             N(e)=k_2 \ell_3 e\\
             N(f)=(k_1 \ell_3+k_2 \ell_2)e+k_2 \ell_3 f \\
            \end{array}
            \right.~(k_2\neq 0, \l=1), \quad \hbox{or}$$
       $$\left\{
            \begin{array}{l}
             N(e)=0\\
             N(f)=k_3 \ell_2 f \\
            \end{array}
            \right. ~(k_3\neq 0), \quad \hbox{or} \quad \left\{
                  \begin{array}{l}
                    N(e)=0\\
                    N(f)=k_1 \ell_3 e \\
                  \end{array}
                  \right. ~(k_1\neq 0);$$
   \item For case (b) in Example~\mref{ex:sym}:
   $$N=0,\quad \hbox{or} \quad \left\{
            \begin{array}{l}
             N(e)=0\\
             N(f)=k_3 \ell_2 f \\
            \end{array}
            \right. ~(k_3\neq 0), \quad \hbox{or}\hspace{3mm} \left\{
            \begin{array}{l}
             N(e)=0\\
             N(f)=k_2 \ell_2 e \\
            \end{array}
            \right.~(k_2\neq 0, \l=1). $$
 \end{enumerate}
 \end{ex}

\section{Nijenhuis pre-Lie bialgebras}\label{se:rep}
We now combine the algebraic and coalgebraic Nijenhuis structures. Dual representations supply the admissibility conditions, while the coalgebraic constructions provide the second operator needed for a Nijenhuis pre-Lie bialgebra. We then study matched pairs, triangular solutions, and relative Rota-Baxter constructions.

\subsection{Dual representations}\label{se:drep}
The representation-theoretic preliminaries lead to the dual admissibility conditions used below. Let $(A, \ci)$ be a pre-Lie algebra, $V$ a vector space, $\rho, \vp:  A\lr \mathrm{End}(V)$ linear maps. A triple $(V, \rho, \vp)$ is called a {\bf representation of $(A, \ci)$} \cite{Bai2} if for all $x, y\in A$,
 \begin{eqnarray}
 &\rho(x\ci y-y\ci x)v=\rho(x)(\rho(y)v)-\rho(y)(\rho(x)v),&\label{eq:rep1}\\
 &\vp(x\ci y)v=\rho(x)(\vp(y)v)-\vp(y)(\rho(x)v)+\vp(y)(\vp(x)v).&\label{eq:rep2}
 \end{eqnarray}
 Let $(V, \rho, \vp)$ and $(V', \rho', \vp')$ be representations of $(A, \ci)$ and $(A', \ci')$, respectively. A pair $(f, g)$ of linear maps $f: A\lr A'$ and $g: V\lr V'$ is called a {\bf homomorphism from $(V, \rho, \vp)$ to $(V', \rho', \vp')$} if for all $x, y\in A, v\in V$,
 \begin{eqnarray*}
 &f(x\ci y)=f(x)\ci' f(y),\quad \rho'(f(x))g(v)=g(\rho(x)v),\quad \vp'(f(x))g(v)=g(\vp(x)v).&
 \end{eqnarray*}

 \begin{defi}\label{de:repnplie} Let $((A, \ci), N)$ be a \N pre-Lie algebra, $(V, \rho, \vp)$ a representation of $(A, \ci)$ and $\a: V\lr V$ a linear map. Then a pair $((V, \rho, \vp), \a)$ is called a {\bf representation of $((A, \ci), N)$} if for all $x\in  A$ and $v\in V$,
 \begin{eqnarray}
 &\rho(N(x))\a(v)+\a^2(\rho(x)v)=\a(\rho(N(x))v)+\a(\rho(x)\a(v)),&\label{eq:nrep1}\\
 &\vp(N(x))\a(v)+\a^2(\vp(x)v)=\a(\vp(N(x))v)+\a(\vp(x)\a(v)).&\label{eq:nrep2}
 \end{eqnarray}

 A homomorphism $(f, g)$ from $(V, \rho, \vp)$ to $(V', \rho', \vp')$ is called a {\bf homomorphism from $((V, \rho, \vp), \a)$ to $((V', \rho', \vp'), \a')$} if for all $x\in A$ and $v\in V$, $N'(f(x))=f(N(x))$ and $g(\a(v))=\a'(g(v))$.
 \end{defi}

\begin{ex}\label{ex:adrepnplie} Let $((A, \ci), N)$ be a \N pre-Lie algebra and $x\in A$. Then $((A, L, R), N)$ is a representation of $((A, \ci), N)$, called the {\bf regular representation of $((A, \ci), N)$}.
 \end{ex}

A \rep of a \N pre-Lie algebra has the following description by a semi-direct product.

 \begin{pro}\label{pro:repnplie} Let $((A, \ci), N)$ be a \N pre-Lie algebra, $(V, \rho, \vp)$ be a representation of $(A, \ci)$ and $\a: V\lr V$ a linear map. Define a product $\ci_{\heartsuit}$ on $ A\oplus V$ by
 \begin{eqnarray}
 &(x+u)\ci_{\heartsuit} (y+v):=x\ci y+\rho(x)v+\vp(y)u, \quad x, y\in A, u, v\in V.&  \label{eq:semip}
 \end{eqnarray}
 and a linear map $N_{ A\oplus V}:  A\oplus V\to   A\oplus V$ by
 \begin{eqnarray}
 &N_{ A\oplus V}(x+v):=(N+\a)(x+v)=N(x)+\a(v).&\label{eq:nsemip}
 \end{eqnarray}
 Then $ A\oplus V$ equipped with the product \eqref{eq:semip} and the linear map \eqref{eq:nsemip} is a \N pre-Lie algebra, denoted by $(A\ltimes_{\rho, \vp} V, N+\a)$, if and only if $((V, \rho, \vp), \a)$ is a representation of $((A, \ci), N)$. In this case, we call $(A\ltimes_{\rho, \vp} V, N+\a)$ the {\bf semi-direct product \N pre-Lie algebra.}
 \end{pro}

 \begin{proof} By \cite[Proposition 3.1]{Bai2}, we see that $(A\oplus V, \ci_{\heartsuit}):= A\ltimes_{\rho, \vp} V$ is a pre-Lie algebra if and only if $(V, \rho, \vp)$ is a representation of $(A, \ci)$. On the other hand, by Eq.\,(\ref{eq:nsemip}),
 we know that $N_{ A\oplus V}$ is a \N operator on $ A\ltimes_{\rho, \vp} V$ if and only if Eqs.~(\ref{eq:nrep1}) and (\ref{eq:nrep2}) hold. Thus the proof is completed.
 \end{proof}

 \begin{lem}\label{lem:dualrep} Let $((A, \ci), N)$ be a \N pre-Lie algebra, $(V, \rho, \vp)$ a representation of $(A, \ci)$ and $\b: V\lr V$ a linear map. Define linear maps $\rho^*, \vp^*:  A\lr End(V^*)$ by
$$\langle \rho^*(x) v^*, v \rangle =-\langle v^*, \rho(x)v \rangle,\quad \langle \vp^*(x) v^*, v \rangle =-\langle v^*, \vp(x)v \rangle, \quad x\in  A, v^*\in V^*, v\in V.$$
Then $((V^*, \rho^*-\vp^*, -\vp^*), \b^*)$ is a representation of $((A, \ci), N)$ if and only if for all $x\in  A$ and $v\in V$,
 \begin{eqnarray}
 &\b(\rho(N(x))v)+\rho(x)\b^2(v)=\rho(N(x))\b(v)+\b(\rho(x)\b(v)),&\label{eq:dual3}\\
 &\b(\vp(N(x))v)+\vp(x)\b^2(v)=\vp(N(x))\b(v)+\b(\vp(x)\b(v)).&\label{eq:dual2}
 \end{eqnarray}
 \end{lem}

 \begin{proof} By \cite[Proposition 3.3]{Bai2}, $(V^*, \rho^*-\vp^*, -\vp^*)$ is a representation of $(A, \ci)$. Furthermore, Eq.\,(\ref{eq:nrep2}) holds for $-\vp^*$ if and only if Eq.\,(\ref{eq:dual2}) holds. Eq.\,(\ref{eq:nrep1}) holds for $\rho^*-\vp^*$ if and only if for all $x\in  A$ and $v\in V$,
 \begin{eqnarray*}
 &&\b(\vp(N(x))v)-\b(\rho(N(x))v)+\vp(x)\b^2(v)-\rho(x)\b^2(v)\\
 &&\qquad=\vp(N(x))\b(v)-\rho(N(x))\b(v)+\b(\vp(x)\b(v))-\b(\rho(x)\b(v)),
 \end{eqnarray*}
 which is equivalent to Eq.\,(\ref{eq:dual3}) by Eq.\,(\ref{eq:dual2}). Then the proof is completed.
 \end{proof}

 Next we apply the above result to the regular representation.
 \begin{cor}\label{cor:lem:dualrep} Let $((A, \ci), N)$ be a \N pre-Lie algebra. For a linear map $S:  A\lr  A$, the pair $((A^*, L^*-R^*, -R^*), S^*)$ is a representation of $((A, \ci), N)$ if and only if for all $x, y\in  A$,
 \begin{eqnarray}
 &S(N(x)\ci y)+x\ci S^2(y)=N(x)\ci S(y)+S(x\ci S(y)),&\label{eq:dual3-1}\\
 &S(x\ci N(y))+S^2(x)\ci y=S(x)\ci N(y)+S(S(x)\ci y).&\label{eq:dual2-1}
 \end{eqnarray}
 \end{cor}

 \begin{defi}\label{de:admop} Let $((A, \ci), N)$ be a \N pre-Lie algebra, $(V, \rho, \vp)$ be a representation of $(A, \ci)$ and $\b: V\lr V$ be a linear map. The \N pre-Lie algebra $((A, \ci), N)$ is called {\bf $\b$-admissible with respect to $(V, \rho, \vp)$} if Eqs.~(\ref{eq:dual3}) and (\ref{eq:dual2}) hold. When Eqs.~(\ref{eq:dual3-1}) and (\ref{eq:dual2-1}) hold, we say that $((A, \ci), N)$ is {\bf $S$-admissible}.
 \end{defi}

 \begin{rmk}\label{rmk:dualrep}
  \begin{enumerate}
    \item $((A, \ci), N)$ is not naturally $\a$-admissible with respect to $(V, \rho, \vp)$ even if $((V, \rho, \vp), \a)$ is a representation of $((A, \ci), N)$.
    \item $((A, \ci), N)$ is $N^*$-admissible with respect to $(A^*, L^*-R^*, -R^*)$.
 \end{enumerate}
 \end{rmk}

The dual representation conditions describe the algebraic side of the desired compatibility. We now introduce the coalgebra-side notions, including their corepresentations.

\subsection{Nijenhuis pre-Lie coalgebras and corepresentations}\label{se:cdrep}\label{se:cln}
Dual to the algebraic notion, we first define a Nijenhuis operator on a pre-Lie coalgebra.
 \begin{defi}\label{de:nc}  A {\bf \N pre-Lie coalgebra} is a pair $((A, \D), S)$, where $(A, \D)$ is a pre-Lie coalgebra and $S: A\lr  A$ a linear map such that
 \begin{eqnarray}
 &S(x_{(1)})\o S(x_{(2)})+S^2(x)_{(1)}\o S^2(x)_{(2)}=S(S(x)_{(1)})\o S(x)_{(2)}+S(x)_{(1)}\o S(S(x)_{(2)}),~~x\in A.&\label{eq:nc}
 \end{eqnarray}
 \end{defi}

\N pre-Lie coalgebra can be derived by linearly compatible pre-Lie coalgebra structures, as the coalgebra version of linearly compatible algebra structures \mcite{ZGG1,ZGG2}.

\begin{defi}\label{de:comprelieco} Let $(A, \D)$ and $(A, \d)$ be pre-Lie coalgebras. For fixed $s, t\in K$, define the coproduct $\D_{s,t}$ on $A$ by
 \begin{eqnarray}
 &\D_{s,t}(x):=s x_{(1)}\o x_{(2)}+t x_{[1]}\o x_{[2]},~\forall~s, t\in K, \quad x\in A, &\label{eq:comprelieco}
 \end{eqnarray}
 where $\D(x)=x_{(1)}\o x_{(2)}$ and $\d(x)=x_{[1]}\o x_{[2]}$. If $(A, \D_{s, t})$ is a pre-Lie coalgebra for all $s, t\in K$, then we call {\bf $(A, \D)$ and $(A, \d)$ linearly compatible}.
 \end{defi}

 \begin{lem} \label{lem:comcoalg} Pre-Lie coalgebras $(A, \D)$ and $(A, \d)$ are linearly compatible if and only if, for all $x\in  A$, the equation below holds:
 \begin{eqnarray*}
 &&\hspace{-6mm}x_{(1)[1]}\o x_{(1)[2]}\o x_{(2)}+x_{[1](1)}\o x_{[1](2)}\o x_{[2]}
 -x_{(1)}\o x_{(2)[1]}\o x_{(2)[2]}-x_{[1]}\o x_{[2](1)}\o x_{[2](2)}\nonumber\\
 &&=x_{(1)[2]}\o x_{(1)[1]}\o x_{(2)}+x_{[1](2)}\o x_{[1](1)}\o x_{[2]}-x_{(2)[1]}\o x_{(1)}\o x_{(2)[2]}-x_{[2](1)}\o x_{[1]}\o x_{[2](2)}.
 \end{eqnarray*}
 \end{lem}

 \begin{proof} It is direct by checking Eq.\,(\ref{eq:cl}) for $\D_{s, t}$.
 \end{proof}

\begin{defi}\label{de:corep} Let $(A, \D)$ be a pre-Lie coalgebra. A {\bf corepresentation of $(A, \D)$} is a triple $(V, \xi, \eta)$, where $V$ is a vector space, and $\xi, \eta: V\lr  A\o V$ (write $\xi(v)=v_{-1}\o v_{0}, \eta(v)=v^{-1}\o v^{0}$) are linear maps, such that the following equalities hold for all $v\in V$.
 \begin{eqnarray}
 &v_{-1(1)}\o v_{-1(2)}\o v_{0}-v_{-1(2)}\o v_{-1(1)}\o v_{0}=v_{-1}\o v_{0-1}\o v_{00}-v_{0-1}\o v_{-1}\o v_{00},&\label{eq:crep1}\\
 &{v^{-1}}_{(1)}\o {v^{-1}}_{(2)}\o v^{0}=v_{-1}\o {v_{0}}^{-1}\o {v_{0}}^{0}-{v^0}_{-1}\o v^{-1}\o {v^0}_{0}+v^{0-1}\o v^{-1}\o v^{00}.&\label{eq:crep2}
 \end{eqnarray}

 Let $(V, \xi, \eta)$ and $(V', \xi', \eta')$ be corepresentations of $(A, \D)$ and $(A', \D')$, respectively. A {\bf homomorphism from $(V, \xi, \eta)$ to $(V', \xi', \eta')$} is a pair $(f, g)$ of linear maps $f:  A\lr  A'$ and $g: V\lr V'$, such that for all $x\in  A$ and $v\in V$,
 \begin{eqnarray}
 &\D'(f(x))=(f\o f)(\D(x)),\quad \xi'(g(v))=(f\o g)(\xi(v)),\quad  \eta'(g(v))=(f\o g)(\eta(v)).&\label{eq:homocorep}
 \end{eqnarray}
 \end{defi}

\begin{lem}\label{lem:comcorep} Let $(A, \D)$ and $(A, \d)$ be linearly compatible pre-Lie coalgebras, and $\xi, \eta, \g, \G: V\lr  A\o V$ (write $\xi(v)=v_{-1}\o v_0, \eta(v)=v^{-1}\o v^0, \g(v)=v_{(-1)}\o v_{(0)}$ and $\G(v)=v^{(-1)}\o v^{(0)}$) be four linear maps. Set
$$\xi_{s, t}:=s \xi+t \g, \quad \eta_{s, t}:=s \eta+t \G, \quad s, t\in K.$$
Then $(V, \xi_{s, t}, \eta_{s, t})$ is a corepresentation of $(A, \D_{s, t})$ for all $s, t\in K$ if and only if $(V, \xi, \eta)$ is a co\rep of $(A, \D)$, $(V, \g, \G)$ is a corepresentation of $(A, \d)$, and the equations below hold for all $v\in V$:
 \begin{eqnarray*}
 &&\hspace{-6mm}v_{-1[1]}\o v_{-1[2]}\o v_{0}+v_{(-1)(1)}\o v_{(-1)(2)}\o v_{(0)}-v_{-1[2]}\o v_{-1[1]}\o v_{0}-v_{(-1)(2)}\o v_{(-1)(1)}\o v_{(0)}\\
 &&\quad=v_{-1}\o {v_{0}}^{(-1)}\o {v_{0}}^{(0)}+v_{(-1)}\o {v_{(0)}}^{-1}\o {v_{(0)}}^{0}-{v_{0}}^{(-1)}\o v_{-1}\o {v_{0}}^{(0)}-{v_{(0)}}^{-1}\o v_{(-1)}\o {v_{(0)}}^{0},\\
 &&\hspace{-6mm}{v^{-1}}_{[1]}\o {v^{-1}}_{[2]}\o {v^{0}}+{v^{(-1)}}_{(1)}\o {v^{(-1)}}_{(2)}\o {v^{(0)}}+{v^{0}}_{(-1)}\o {v^{-1}}\o {v^{0}}_{(0)}+{v^{(0)}}_{-1}\o {v^{(-1)}}\o {v^{(0)}}_{0}\\
 &&\quad=v_{-1}\o {v_{0}}^{(-1)}\o {v_{0}}^{(0)}+v_{(-1)}\o {v_{(0)}}^{-1}\o {v_{(0)}}^{0}+v^{0(-1)}\o v^{-1}\o v^{0(0)}+v^{(0)-1}\o v^{(-1)}\o v^{(0)0}.
 \end{eqnarray*}
 \end{lem}

 \begin{proof} It is direct by checking Eqs.~(\ref{eq:crep1}) and (\ref{eq:crep2}) for $\D_{s, t}$, $\xi_{s, t}$ and $\eta_{s, t}$.
 \end{proof}

 \begin{thm}\label{thm:comcn} Let $(A, \D)$ and $(A, \d)$ be linearly compatible pre-Lie coalgebras. Let $(V, \xi, \eta)$ be a corepresentation of $(A, \D)$, $(V, \xi_{s, t}, \eta_{s, t})$ be a corepresentation of $(A, \D_{s, t})$ for $s,t \in K$, and $S:  A\lr  A, \b: V\lr V$ be two linear maps.
\begin{enumerate}
\item \label{i:comcn1}
Then $(s\id_ A+t S, s\id_V+t \b)$ is a corepresentation homomorphism from $(V, \xi, \eta)$ to $(V, \xi_{s, t}, \eta_{s, t})$ for all $s, t\in K$ if and only if for all $x\in  A$ and $v\in V$, the equations below hold:
 \begin{eqnarray}
 &S(x)_{[1]}\o S(x)_{[2]}=S(x_{(1)})\o S(x_{(2)}),&\label{eq:com1}\\
 &S(x)_{(1)}\o S(x)_{(2)}+x_{[1]}\o x_{[2]}=x_{(1)}\o S(x_{(2)})+S(x_{(1)})\o x_{(2)},&\label{eq:com2}\\
 &\b(v)_{(-1)}\o \b(v)_{(0)}=S(v_{-1})\o \b(v_{0}),&\label{eq:com3}\\
 &\b(v)_{-1}\o \b(v)_{0}+v_{(-1)}\o v_{(0)}=v_{-1}\o \b(v_{0})+S(v_{-1})\o v_{0}.&\label{eq:com4}\\
 &\b(v)^{(-1)}\o \b(v)^{(0)}=S(v^{-1})\o \b(v^{0}),&\label{eq:com5}\\
 &\b(v)^{-1}\o \b(v)^{0}+v^{(-1)}\o v^{(0)}=v^{-1}\o \b(v^{0})+S(v^{-1})\o v^{0}.&\label{eq:com6}
 \end{eqnarray}
\item \label{i:comcn2} If one of the equivalent conditions in \eqref{i:comcn1} holds, then  $((A, \D), S)$ is a \N pre-Lie coalgebra, and {\bf $((V, \xi, \eta), \b)$ is a co\rep of $((A, \D), S)$}, in the sense that, $(V, \xi, \eta)$ is a co\rep of $(A, \D)$, and for all $v\in V$, the following equations hold:
\begin{eqnarray}
	&S(v_{-1})\o \b(v_{0})+\b^2(v)_{-1}\o \b^2(v)_{0}=S(\b(v)_{-1})\o \b(v)_{0}+\b(v)_{-1}\o \b(\b(v)_{0}),&\label{eq:ncrep1}\\
	&S(v^{-1})\o \b(v^{0})+\b^2(v)^{-1}\o \b^2(v)^{0}=S(\b(v)^{-1})\o \b(v)^{0}+\b(v)^{-1}\o \b(\b(v)^{0}),&\label{eq:ncrep2}
\end{eqnarray}
where $\xi(v)=v_{-1}\o v_0, \eta(v)=v^{-1}\o v^0$.
\end{enumerate}
 \end{thm}
\begin{proof} The proof of \eqref{i:comcn1} is a routine though tedious check of Eq.\,(\ref{eq:homocorep}) for $(s\id_ A+t S, s\id_V+t \b)$.
The first statement of \eqref{i:comcn2} follows from Eqs.~(\mref{eq:com1}) and (\mref{eq:com2}).
The second statement follows from Eqs.~(\mref{eq:com3})-(\mref{eq:com6}).
\end{proof}

\begin{defi} Let $((V, \xi, \eta), \b)$ and $((V', \xi', \eta'), \b')$ be corepresentations of $((A, \D), S)$ and $((A', \D'),$ $S')$, respectively. A {\bf homomorphism from $((V, \xi, \eta), \b)$ to $((V', \xi', \eta'), \b')$} is a pair $(f, g)$ of homomorphisms from $(V, \xi, \eta)$ and $(V', \xi', \eta')$ such that for $S'\circ f=f \circ S$ and $\b'\circ g=g\circ\b$.
 \end{defi}

With these coalgebraic preliminaries in place, we can formulate the compatibility conditions between an algebraic Nijenhuis operator and its coalgebraic counterpart.

\subsection{Nijenhuis pre-Lie bialgebras and their characterizations} \label{se:equi}
We now introduce the central bialgebraic structure.
 \begin{defi}\label{de:nliebialg} A {\bf \N pre-Lie bialgebra} is a triple $((A, \ci, \D), N, S)$, where
  \begin{enumerate}
    \item \label{it:nliebialg1} $((A, \ci), N)$ is a \N pre-Lie algebra;
    \item \label{it:nliebialg2} $((A, \D), S)$ is a \N pre-Lie coalgebra;
    \item \label{it:nliebialg3} $(A, \ci, \D)$ is a pre-Lie bialgebra;
    \item \label{it:nliebialg4} $((A, \ci), N)$ is $S$-admissible, i.e., Eqs.~(\ref{eq:dual3-1}) and (\ref{eq:dual2-1}) hold;
    \item \label{it:nliebialg5} $((A^*, \D^*), S^*)$ is $N^*$-admissible, i.e., for all $x\in  A$,
         \begin{eqnarray}
           &N(x)_{(1)}\o S(N(x)_{(2)})+N^2(x_{(1)})\o x_{(2)}=N(x_{(1)})\o S(x_{(2)})+N(N(x)_{(1)})\o N(x)_{(2)},&\label{eq:nliebialg1}\\
           &S(N(x)_{(1)})\o N(x)_{(2)}+x_{(1)}\o N^2(x_{(2)})=S(x_{(1)})\o N(x_{(2)})+N(x)_{(1)}\o N(N(x)_{(2)}).&\label{eq:nliebialg2}
         \end{eqnarray}
 \end{enumerate}
 \end{defi}

 \begin{ex}\label{ex:nliebialg} Let $(A, \ci)$ be the pre-Lie algebra given in Example \ref{ex:sym}, and $\l, k_1, k_2, k_3, \ell_1, \ell_2, \ell_3\in K$ be parameters.
 \begin{enumerate}[(I)]
    \item Assume $k_2=0$ or $\lambda=1$. Set
   \begin{center}$ \left\{
        \begin{array}{l}
        \D(e):=k_2 e\o e,\\
        \D(f):=-2 k_1 e\o e-k_2 e\o f
        \end{array}\right. $
             \end{center}
        Then $(A, \D)$ is a pre-Lie coalgebra and further $(A, \ci, \D)$ is a pre-Lie bialgebra. Define linear maps $N, S: A\lr  A$ by:
      \begin{center} $\left\{
        \begin{array}{l}
         N(e):=k_2 \ell_3 e\\
         N(f):=(k_1 \ell_3+k_2 \ell_2)e+k_2 \ell_3 f\\
         \end{array}
           \right.$,\quad $S=N$.
             \end{center}
 Then by a direct computation we find that $((A, \ci, \D), N, S)$ is a \N pre-Lie bialgebra.
    \item Set
   \begin{center}$ \left\{
        \begin{array}{l}
        \D(e):=k_3 f\o e,\\
        \D(f):=\l k_3 f\o f
        \end{array}\right. $
             \end{center}
        Then $(A, \D)$ is a pre-Lie coalgebra and further $(A, \ci, \D)$ is a pre-Lie bialgebra. Define linear maps $N, S: A\lr  A$ by:
      \begin{center} $\left\{
        \begin{array}{l}
         N(e):=0\\
         N(f):=k_3 \ell_2 f\\
         \end{array}
           \right.$,\quad $\left\{
        \begin{array}{l}
         S(e):=k_3 \ell_2 e\\
         S(f):=k_3 \ell_2 f\\
         \end{array}
           \right.$
             \end{center}
   Then $((A, \ci, \D), N, S)$ is a \N pre-Lie bialgebra.
 \end{enumerate}
 \end{ex}

\subsubsection{A construction of Nijenhuis operators on pre-Lie coalgebras}\label{se:connprelieco}
We next obtain the coalgebra-side counterpart of the construction in Section \ref{se:suff}. The starting point is a weak pseudo-Hessian pre-Lie coalgebra, dual to the structure in Definition \mref{de:sym}.
 \begin{defi}\label{de:csym} Let $(A, \D)$ be a pre-Lie coalgebra and $r=\sum\limits^n_{i=1} a_i\o b_i$ be a symmetric element in $A\o A$. Assume that
 \begin{eqnarray}\label{eq:csymp1}
 \sum\limits_{i=1}\limits^n ({a_i}_{(1)}\o {a_i}_{(2)}\o b_i-a_i\o {b_i}_{(1)}\o {b_i}_{(2)})=\sum\limits_{i=1}\limits^n({a_i}_{(2)}\o {a_i}_{(1)}\o b_i-{b_i}_{(1)}\o a_i\o {b_i}_{(2)}).
 \end{eqnarray}
 Then we call $(A, \D)$ a {\bf weak pseudo-Hessian pre-Lie coalgebra} and denote it by $(A, \D, r)$.
 \end{defi}

The following theorem is dual to Theorem \ref{thm:ln}. Although the construction follows the same guiding idea, we include the proof to make the coalgebraic argument explicit.

 \begin{thm}\label{thm:cln} Let $(A, \D, r)$ be a weak pseudo-Hessian pre-Lie coalgebra and $\om\in (A\o  A)^*$. For all $x\in A$, define a linear map $S: A\lr  A$ by
 \begin{eqnarray}
 &S(x):=\sum\limits_{i=1}\limits^n a_i\om(b_i, x).&\label{eq:cys}
 \end{eqnarray}
 If $(A, \D, \om, \ci_\om)$ is a dual triangular pre-Lie bialgebra with the multiplication $\ci_\om$ defined in Eq.~\meqref{eq:p}, and additionally, $(A, \ci_\om, r, \D_r)$ is a triangular pre-Lie bialgebra with the comultiplication $\D_r$ defined in Eq.~\meqref{eq:cop}, then $((A, \D), S)$ is a \N pre-Lie coalgebra.
 \end{thm}

\begin{proof}
Since $(A,\ci_\om,r,\D_r)$ is triangular, $r$ satisfies the $\mathbb{S}$-equation \eqref{eq:cybe} for $\ci_\om$. Substituting \eqref{eq:p}, using the symmetry of $r$, and collecting terms with the same tensor factors gives, by the weak pseudo-Hessian identity \eqref{eq:csymp1},
{\small\begin{eqnarray}\label{eq:cqt-sym1}
 &&\sum\limits_{i,j=1}\limits^n \Big({a_i}_{(1)}\o {a_i}_{(2)}\om(b_i, a_j)\o b_j-{a_i}_{(2)}\o {a_i}_{(1)}\om(b_i, a_j)\o b_j \\
 &&\qquad\qquad\qquad-a_i\o {b_j}_{(1)}\om(b_i, a_j) \o {b_j}_{(2)}+{b_i}_{(1)}\om(a_j, a_i)\o b_j \o {b_i}_{(2)}\Big)=0.\nonumber
\end{eqnarray}}
We verify that $S$ in \eqref{eq:cys} satisfies the Nijenhuis coalgebra identity. Denote its left-hand side, moved to one side, by
$$
\mathcal T_S(x)=(S\o S)\D(x)-(S\o\id)\D(Sx)-(\id\o S)\D(Sx)+\D(S^2x).
$$
Substitution of \eqref{eq:cys} into $\mathcal T_S(x)$ produces four double sums. Applying \eqref{eq:cqt-sym1} to the term containing $\D(a_i)$ and then \eqref{eq:csymp1} to the remaining terms reduces the expression to contractions of the co-$\mathbb{S}$-equation \eqref{eq:ccybe} with $r$ and $x$. Consequently $\mathcal T_S(x)=0$ for every $x\in A$, proving that $S$ is a Nijenhuis operator on $(A,\D)$.
\end{proof}

\begin{rmk}
 Theorem \ref{thm:cln} is the formal dual of Theorem \ref{thm:ln} in the sense that it is obtained by reversing all arrows (replacing the pre-Lie algebra $(A,\ci)$ with a pre-Lie coalgebra $(A, \D)$, symmetric bilinear forms $\omega\in (A\otimes A)^*$ with symmetric elements $r\in A\otimes A$, and the multiplication $\ci_\omega$ with the comultiplication $\D_r$). 
\end{rmk}

Theorems \ref{thm:ln} and \ref{thm:cln} thus provide the algebraic and coalgebraic operators separately. The next result identifies when they combine to define a Nijenhuis pre-Lie bialgebra.

\begin{thm}\label{thm:n-cn} Let $(A, \ci, \D)$ be a pre-Lie bialgebra. Assume that $\om\in (A\o A)^*$ and $r=\sum\limits^m_{i=1} a_i\o b_i\in  A\o A$ satisfy the condition of Theorem \mref{thm:ln}, $\varpi\in (A\o A)^*$ and $\mathbf{r}=\sum\limits^n_{j=1} c_j\o d_j\in  A\o A$ satisfy the condition of Theorem \mref{thm:cln}. Then $((A, \ci, \D), N, S)$ is a \N pre-Lie bialgebra if and only if
 \begin{eqnarray*}
 &&\hspace{-2.5mm}\sum\limits_{i,j}\limits \Big(\om(x, a_i)\big(c_j\varpi(d_j, b_i\ci y)-b_i\ci c_j\varpi(d_j, y)\big)+\big(x\ci c_i\varpi(d_i, c_j)-c_i\varpi(d_i, x\ci c_j)\big)\varpi(d_j, y)\Big)=0,\\
 &&\hspace{-2.5mm}\sum\limits_{i,j}\limits \Big(\om(y, a_i)\big(c_j\varpi(d_j, x\ci b_i)-c_j\ci b_i\varpi(d_j, x)\big)+\big(c_i\ci y\varpi(d_i, c_j)-c_i\varpi(d_i, c_j\ci y)\big)\varpi(d_j, x)\Big)=0,\\
 &&\hspace{-2.5mm}\sum\limits_{i,j}\limits \Big(\om(x, a_i)\big(b_{i(1)}\o c_j\varpi(d_j, b_{i(2)})-\om(b_{i(1)}, a_j)b_j\o b_{i(2)}\big)\\
 &&\hspace{60mm}+\om(x_{(1)}, a_i)\big(\om(b_i, a_j)b_j\o x_{(2)}-b_i\o c_j\varpi(d_j, x_{(2)})\big)\Big)=0,\\
 &&\hspace{-2.5mm}\sum\limits_{i,j}\limits \Big(\om(x, a_i)\big(c_j\o b_{i(2)}\varpi(d_j, b_{i(1)})-\om(b_{i(2)}, a_j)b_{i(1)}\o b_j\big)\\
 &&\hspace{60mm}+\om(x_{(2)}, a_i)\big(\om(b_i, a_j)x_{(1)}\o b_j-c_j\o b_i\varpi(d_j, x_{(1)})\big)\Big)=0.
 \end{eqnarray*}
 \end{thm}

\begin{proof}
Theorems \ref{thm:ln} and \ref{thm:cln} guarantee that, with $N$ given by Eq.\,(\mref{eq:ys}) and $S$ given by Eq.\,(\mref{eq:cys}), the pair  $((A, \ci), N)$ is a \N pre-Lie algebra and $((A, \D), S)$ is a \N pre-Lie coalgebra. Then the proof is completed by observing that the equations in the theorem are obtained from replacing $N$ and $S$ respectively by their expressions in Theorems \ref{thm:ln} and \ref{thm:cln}.
\end{proof}

 We now turn from the construction criterion to a structural characterization. Pre-Lie bialgebras admit a description by matched pairs of pre-Lie algebras; we extend the matched-pair framework of \cite[Theorem 3.5]{Bai2} to the \N case.

 \begin{defi}\label{de:mpnlie} A {\bf matched pair of Nijenhuis pre-Lie algebras} is a sextuple $\big(((A, \ci_ A), N_ A), ((\hh,$ $\ci_\hh), N_\hh), \rho_ A, \vp_ A, \rho_\hh, \vp_\hh\big)$, where
 \begin{enumerate}
 \item
 $((A, \ci_ A), N_ A)$ and $((\hh, \ci_\hh), N_\hh)$ are \N pre-Lie algebras,
 \item
 $\big((\hh, \rho_ A, \vp_ A), N_\hh\big)$ is a representation of $\big((A, \ci_ A), N_ A\big)$,
 \item $\big((A, \rho_\hh, \vp_\hh), N_ A\big)$ is a representation of $\big((\hh, \ci_\hh), N_\hh\big)$, and
 \item $\big((A, \ci_ A), (\hh, \ci_\hh), \rho_ A, \vp_ A, \rho_\hh, \vp_\hh\big)$ is a matched pair of pre-Lie algebras $(A, \ci_ A)$ and $(\hh, \ci_\hh)$.
\end{enumerate}
Here the last condition \cite{Bai2} means that the following equations hold for all $x, y\in  A$ and $a, b\in \hh$:
{\small \begin{eqnarray}
 &\rho_ A(x)(a\ci_\hh b)=-\rho_ A(\rho_\hh(a)x-\vp_\hh(a)x)b+(\rho_ A(x)a-\vp_ A(x)a)\ci_\hh b+\vp_ A(\vp_\hh(b)x)a+a\ci_\hh (\rho_ A(x)b),&\label{eq:mplie1}\\
 &\vp_ A(x)(a\ci_\hh b-b\ci_\hh a)=\vp_ A(\rho_\hh(b)x)a-\vp_ A(\rho_\hh(a)x)b+a\ci_\hh (\vp_ A(x)b)-b\ci_\hh (\vp_ A(x)a),&\label{eq:mplie2}\\
 &\rho_\hh(a)(x\ci_ A y)=-\rho_\hh(\rho_ A(x)a-\vp_ A(x)a)y+(\rho_\hh(a)x-\vp_\hh(a)x)\ci_ A y+\vp_\hh(\vp_ A(y)a)x+x\ci_ A (\rho_\hh(a)y),&\label{eq:mplie3}\\
 &\vp_\hh(a)(x\ci_ A y-y\ci_ A x)=\vp_\hh(\rho_ A(y)a)x-\vp_\hh(\rho_ A(x)a)y+x\ci_ A (\vp_\hh(a)y)-y\ci_ A (\vp_\hh(a)x).&\label{eq:mplie4}
 \end{eqnarray}
}
 \end{defi}

 \begin{pro}\label{pro:mpnlie} Let $((A, \ci_ A), N_ A)$ and $((\hh, \ci_\hh), N_\hh)$ be Nijenhuis pre-Lie algebras. Assume that $((\hh, \rho_ A, \vp_ A),$ $N_\hh)$ is a representation of $((A, \ci_ A), N_ A)$ and $((A, \rho_\hh, \vp_\hh), N_ A)$ is a representation of $((\hh, \ci_\hh), N_\hh)$. Define a product $\ci_{\star}$ on $ A\oplus \hh$ by
 \begin{eqnarray*}
 &(x+a)\ci_\star (y+b)=x\ci_ A y+\rho_\hh(a)y+\vp_\hh(b)x+a\ci_\hh b+\rho_ A(x)b+\vp_ A(y)a,&
 \end{eqnarray*}
 and a linear map $N_\star$ on $ A\oplus \hh$ by
 \begin{eqnarray*}
 &N_\star(x+a):=N_ A(x)+N_\hh(a),\quad x, y\in  A, a, b\in \hh. &
 \end{eqnarray*}
Then $((A\oplus \hh, \ci_{\star}), N_\star)$ is a \N pre-Lie algebra if and only if Eqs.~\eqref{eq:mplie1}--\eqref{eq:mplie4} hold. In this case, we denote $((A\oplus \hh, \ci_{\star}), N_\star)$ by $(A\bowtie \hh, N_ A+N_\hh)$.
 \end{pro}

 \begin{proof} Based on \cite[Theorem 3.5]{Bai2}, $(A\oplus \hh, \ci_{\star})$ is a pre-Lie algebra if and only if Eqs.~(\ref{eq:mplie1})--(\ref{eq:mplie4}) hold. So we only need to check that $N_\star$ satisfies Eq.\,(\ref{eq:n}) if and only if Eqs.\meqref{eq:nrep1} and \meqref{eq:nrep2} hold. In fact, for all $x, y\in  A$ and $a, b\in \hh$, we have
 \begin{eqnarray*}
 &&\hspace{-6mm}N_\star(x+a)\ci_\star N_\star(y+b)+N_\star^2((x+a)\ci_\star (y+b))\\
 &&=N_ A(x)\ci_ A N_ A(y)+\rho_\hh(N_\hh(a))N_ A(y)+\vp_\hh(N_\hh(b))N_ A(x)+N_\hh(a)\ci_\hh N_\hh(b)+\rho_ A(N_ A(x))N_\hh(b)\\
 &&\quad+\vp_ A(N_ A(y))N_\hh(a)+N_ A^2(x\ci_ A y)+N_ A^2(\rho_\hh(a)y)+N_ A^2(\vp_\hh(b)x)+N_\hh^2(a\ci_\hh b)\\
 &&\quad+N_\hh^2(\rho_ A(x)b)+N_\hh^2(\vp_ A(y)a),\\
&&\hspace{-6mm}
N_\star\bigl(N_\star(x+a)\circ_\star(y+b)
                     +(x+a)\circ_\star N_\star(y+b)\bigr)\\
&&={}N_A\Bigl(N_A(x)\circ_A y+x\circ_A N_A(y)
  +\rho_{\mathfrak h}(N_{\mathfrak h}(a))y
  +\rho_{\mathfrak h}(a)N_A(y)
+\vp_{\mathfrak h}(b)N_A(x)
  +\vp_{\mathfrak h}(N_{\mathfrak h}(b))x\Bigr)\\
&&\quad+N_{\mathfrak h}\Bigl(N_{\mathfrak h}(a)\circ_{\mathfrak h}b
  +a\circ_{\mathfrak h}N_{\mathfrak h}(b)
  +\rho_A(N_A(x))b+\rho_A(x)N_{\mathfrak h}(b)
+\vp_A(y)N_{\mathfrak h}(a)
  +\vp_A(N_A(y))a\Bigr).
\end{eqnarray*}
 Then the proof is completed by comparing the two equalities above.
 \end{proof}

 \begin{rmk} \label{rmk:pro:mpnlie} If $(((A, \ci_ A), N_ A), ((\hh, \ci_\hh), N_\hh), \rho_ A, \vp_ A, \rho_\hh, \vp_\hh)$ is a matched pair of Nijenhuis pre-Lie algebras $((A, \ci_ A), N_ A)$ and $((\hh, \ci_\hh), N_\hh)$, then we have a Nijenhuis pre-Lie algebra $(A\bowtie \hh, N_ A+N_\hh)$.
 \end{rmk}

\begin{thm} \label{thm:matchandnlieb} Let $((A, \ci), N)$ and $((A^*, \cdot_\ci:=\D^*), S^*)$ be \N pre-Lie algebras. Then the triple $((A, \ci, \D), N, S)$ is a \N pre-Lie bialgebra if and only if $(((A, \ci), N), ((A^*, \cdot_\ci), S^*), L^*$ $-R^*, -R^*, \mathbb{L}^*-\mathbb{R}^*, -\mathbb{R}^*)$ is a matched pair of Nijenhuis pre-Lie algebras $((A, \ci), N)$ and $((A^*, \cdot_\ci), S^*)$.
 \end{thm}

 \begin{proof} According to \cite[Proposition 4.2]{Bai2}, we find that $((A, \ci), \D)$ is a pre-Lie bialgebra if and only if $((A, \ci)$, $(A^*, \cdot_\ci), L^*-R^*, -R^*, \mathbb{L}^*-\mathbb{R}^*, -\mathbb{R}^*)$ is a matched pair of pre-Lie algebras $(A, \ci)$ and $(A^*, \cdot_\ci)$. The rest of the argument is obtained by Items (\ref{it:nliebialg4}) and (\ref{it:nliebialg5}) in Definition \ref{de:nliebialg}.
 \end{proof}

The matched-pair characterization clarifies the abstract compatibility conditions. We next obtain explicit triangular examples through an enriched $\mathbb{S}$-equation.

\subsection{Enrichment of \sss-equations}\label{se:cybe}
 Let $((A, \ci), N)$ be an $S$-admissible Nijenhuis pre-Lie algebra, $r$ be a symmetric solution of the \sss-equation in $(A, \ci)$. Then by Proposition \ref{pro:qt}, we obtain a triangular pre-Lie bialgebra $(A, \ci, r, \D_r)$, where $\D_r$ is given by Eq.\,(\ref{eq:cop}). Thus in this case, by Definition \ref{de:nliebialg}, $((A, \ci), \D_r, N, S)$ is a \N pre-Lie bialgebra if and only if Eqs.~(\ref{eq:nc}), (\ref{eq:nliebialg1}) and (\ref{eq:nliebialg2}) hold. We will give the equivalent descriptions of these three conditions in terms of $r$.

 \begin{lem} \label{lem:pq} Let $((A, \ci), N)$ be an $S$-admissible Nijenhuis pre-Lie algebra and $\D:=\D_r$ given by Eq.\,(\ref{eq:cop}). Then
 \begin{enumerate}
   \item \label{it:pq1} Eq.~\eqref{eq:nc} holds  if and only if, for all $x\in  A$,
    \begin{eqnarray}
     &&(\id\o S\circ R_x-\id\o R_{S(x)}-\id\o S\circ L_x+\id\o L_{S(x)})((S\o \id-\id\o N)(r)) \label{eq:nc-1}\\
     &&\hspace{45mm}+(S\circ L_x\o \id-L_{S(x)}\o \id)((N\o \id-\id\o S)(r))=0.\nonumber
    \end{eqnarray}
   \item \label{it:pq3} Eq.~\eqref{eq:nliebialg1} holds if and only if, for all $x\in A$,
    \begin{eqnarray}
     &&(\id\o R_{N(x)}-\id\o L_{N(x)}-L_{N(x)}\o \id+\id\o S\circ R_x-\id\o S\circ L_x+N\circ L_x\o \id)\label{eq:nliebialg1-1}\\
     &&\hspace{15mm}((N\o \id-\id\o S)(r))=(\id\o L_x\circ S^2-N^2\o L_x+N^2\o R_x-\id\o R_x\circ S^2)(r).\nonumber
    \end{eqnarray}
   \item \label{it:pq2} Eq.~\eqref{eq:nliebialg2} holds if and only if, for all $x\in A$,
    \begin{eqnarray}
     &&(L_{N(x)}\o \id+\id\o L_{N(x)}-\id\o R_{N(x)}+S\circ L_x\o \id-\id\o N\circ L_x+\id\o N\circ R_x)\label{eq:nliebialg2-1}\\
     &&\hspace{45mm}((S\o \id-\id\o N)(r))=(L_x\circ S^2\o \id-L_x\o N^2)(r).\nonumber
    \end{eqnarray}
 \end{enumerate}
 \end{lem}

\begin{proof} (\ref{it:pq1}) For all $x\in  A$, by Eq.\,(\ref{eq:cop}), we have
\begin{eqnarray*}
 S(x_{(1)})\o S(x_{(2)})\hspace{-3mm}&\stackrel{}{=}&\hspace{-3mm}\sum\limits_{i=1}\limits^n \big(S(x\ci a_i)\o S(b_i)+S(a_i)\o S(x\ci b_i)-S(a_i)\o S(b_i\ci x)\big),\\
 S^2(x)_{(1)}\o S^2(x)_{(2)}\hspace{-3mm}&\stackrel{}{=}&\hspace{-3mm}\sum\limits_{i=1}\limits^n \big(S^2(x)\ci a_i\o b_i+a_i\o S^2(x)\ci b_i-a_i\o b_i\ci S^2(x)\big),\\
 S(S(x)_{(1)})\o S(x)_{(2)}\hspace{-3mm}&\stackrel{}{=}&\hspace{-3mm}\sum\limits_{i=1}\limits^n \big(S(S(x)\ci a_i)\o b_i+S(a_i)\o S(x)\ci b_i-S(a_i)\o b_i\ci S(x)\big)\\
 \hspace{-3mm}&\stackrel{(\ref{eq:dual2-1})}{=}&\hspace{-3mm}\sum\limits_{i=1}\limits^n \big(S(x\ci N(a_i))\o b_i+S^2(x)\ci a_i\o b_i-S(x)\ci N(a_i)\o b_i\\
 &&\hspace{-3mm}+S(a_i)\o S(x)\ci b_i-S(a_i)\o b_i\ci S(x)\big),\\
 S(x)_{(1)}\o S(S(x)_{(2)})\hspace{-3mm}&\stackrel{}{=}&\hspace{-3mm}\sum\limits_{i=1}\limits^n \big(S(x)\ci a_i\o S(b_i)+a_i\o S(S(x)\ci b_i)-a_i\o S(b_i\ci S(x))\big)\\
 \hspace{-3mm}&\stackrel{(\ref{eq:dual3-1})(\ref{eq:dual2-1})}{=}&\hspace{-3mm}\sum\limits_{i=1}\limits^n \big(S(x)\ci a_i\o S(b_i)+a_i\o S(x\ci N(b_i))+a_i\o S^2(x)\ci b_i\\
 &&\hspace{-3mm}-a_i\o S(x)\ci N(b_i)+a_i\o N(b_i)\ci S(x)-a_i\o S(N(b_i)\ci x)\\
 &&\hspace{-3mm}-a_i\o b_i\ci S^2(x)\big).
 \end{eqnarray*}

Thus Eq.\,(\ref{eq:nc}) $\Leftrightarrow$ Eq.\,(\ref{eq:nc-1}).
 \smallskip

Similarly, we can obtain (\ref{it:pq3}) and (\ref{it:pq2}).
\end{proof}

 \begin{rmk} \label{rmk:thm:pq}
 \begin{enumerate}
   \item If $\sum\limits_{i=1}\limits^n \big(S(a_i)\o b_i\big)=\sum\limits_{i=1}\limits^n \big(a_i\o N(b_i)\big)$, then it follows that $\sum\limits_{i=1}\limits^n \big(S^2(a_i)\o b_i\big)=\sum\limits_{i=1}\limits^n \big(a_i\o N^2(b_i)\big)$.
   \item If $\sum\limits_{i=1}\limits^n \big(N(a_i)\o b_i\big)=\sum\limits_{i=1}\limits^n \big(a_i\o S(b_i)\big)$, then it follows that $\sum\limits_{i=1}\limits^n \big(N^2(a_i)\o b_i\big)=\sum\limits_{i=1}\limits^n \big(a_i\o S^2(b_i)\big)$.
   \item If $r$ is symmetric, then it follows that $\sum\limits_{i=1}\limits^n \big(S(a_i)\o b_i\big)=\sum\limits_{i=1}\limits^n \big(a_i\o N(b_i)\big)$ is equivalent to $\sum\limits_{i=1}\limits^n \big(N(a_i)\o b_i\big)=\sum\limits_{i=1}\limits^n \big(a_i\o S(b_i)\big)$.
 \end{enumerate}
 \end{rmk}

  By Lemma \ref{lem:pq} and Remark \ref{rmk:thm:pq}, we obtain

 \begin{thm}\label{thm:nqt} Let $((A, \ci), N)$ be an $S$-admissible Nijenhuis pre-Lie algebra and $\D:=\D_r$ given by Eq.\,(\ref{eq:cop}). If $r$ is a symmetric solution of the \sss-equation in $(A, \ci)$ such that
 \begin{eqnarray}
 &(S\o \id-\id\o N)(r)=0,& \label{eq:cybe-1}
 \end{eqnarray}
 then $((A, \ci), \D_r, N, S)$ is a \N pre-Lie bialgebra. In this case, we call $((A, \ci), \D_r, N, S)$ {\bf triangular}.
 \end{thm}

 \begin{rmk} Let $((A, \ci), N)$ be a Nijenhuis pre-Lie algebra and $S:  A\lr  A$ a linear map. A triangular \N pre-Lie bialgebra $((A, \ci), \D_r, N, S)$ is a triangular pre-Lie bialgebra $((A, \ci), \D_r)$ satisfying Eqs.\,(\ref{eq:dual3-1}), (\ref{eq:dual2-1}) and (\ref{eq:cybe-1}).
 \end{rmk}

 \begin{defi} \label{de:cybe-1} Let $((A, \ci), N)$ be a Nijenhuis pre-Lie algebra. Then the \sss-equation (i.e. Eq.\,(\ref{eq:cybe})) together with Eq.\,(\ref{eq:cybe-1}) is called the {\bf $S$-Nijenhuis \sss-equation}.
 \end{defi}

 \begin{rmk}
 \begin{enumerate}
\item Let $((A, \ci), N)$ be an $S$-admissible Nijenhuis pre-Lie algebra. If $r$ is a symmetric solution of the $S$-Nijenhuis \sss-equation, then we obtain a \N pre-Lie bialgebra by Theorem \ref{thm:nqt}.
\item Let $((A, \ci), N)$ be a Nijenhuis pre-Lie algebra, $r\in A\o A$ be a symmetric element. Applying $\id_A\o \sigma$ to the \sss-equation and by $r=\sigma(r)$, we obtain
    \begin{equation}\label{eq:cybe-2}
    \sum^n_{i,j=1} (a_i\o a_j\o b_i\ci b_j+a_i\o b_i\ci a_j\o b_j)
     =\sum^n_{i,j=1} (a_i\ci a_j\o b_j\o b_i+a_i\o a_j\ci b_i\o b_j).
    \end{equation}
   Thus Eq.\,(\ref{eq:cybe-2}) is equivalent to the \sss-equation.
 \end{enumerate}
 \end{rmk}

\begin{thm} \label{thm:rr} Let $((A, \ci), N)$ be a Nijenhuis pre-Lie algebra, $r$ a symmetric element in $ A\o  A$ and $S: A\lr  A$ a linear map. Then $r$ is a solution of the $S$-Nijenhuis \sss-equation if and only if for all $a^*, b^*\in  A^*$, $\rr$ satisfies
 \begin{eqnarray}\label{eq:rr-1}
 &\rr(a^*)\ci \rr(b^*)=\rr\big((L^*-R^*)(\rr(a^*))(b^*)-R^*(\rr(b^*))(a^*)\big)&
 \end{eqnarray}
 and
 \begin{eqnarray}\label{eq:rr-2}
 N(\rr(a^*))=\rr(S^*(a^*)).
 \end{eqnarray}
 \end{thm}

 \begin{proof} By \cite[Theorem 6.6]{Bai2} or \cite[Proposition 2.6]{WSBL}, we know that $r$ is a solution of the \sss-equation if and only if Eq.\,(\ref{eq:rr-1}) holds. In detail, since $r$ is symmetric, for all $a^*, b^*\in  A^*$, we have
 \begin{eqnarray*}
 Eq.\,(\ref{eq:rr-1})
 &\stackrel{(\ref{eq:rr})}{\Leftrightarrow}&\sum^n_{i,j=1} (\langle a^*, a_i\rangle\langle b^*, a_j\rangle b_i\ci b_j+\langle a^*, a_i\rangle\langle b^*, b_i\ci a_j\rangle b_j)\\
 &&\hspace{3mm}=\sum^n_{i,j=1} (\langle a^*, a_i\rangle\langle b^*, a_j\ci b_i\rangle b_j+\langle a^*, a_j\ci b_i\rangle\langle b^*, a_i\rangle b_j)\\
 &\stackrel{\forall~a^*,~~b^*}{\Leftrightarrow}&\sum^n_{i,j=1} (a_i\o a_j\o b_i\ci b_j+a_i\o b_i\ci a_j\o b_j)\\
 &&\hspace{3mm}=\sum^n_{i,j=1} (a_i\o a_j\ci b_i\o b_j+a_i\ci b_j\o a_j\o b_i)
 \stackrel{}{\Leftrightarrow}  Eq.\,(\ref{eq:cybe-2}).
 \end{eqnarray*}
 Also, one gets
 \begin{eqnarray*}
 \rr(S^*(a^*))&=&\sum\limits_{i=1}\limits^n \big(\langle S^*(a^*), a_i \rangle b_i\big)=\sum\limits_{i=1}\limits^n \big(\langle a^*, S(a_i) \rangle b_i\big),\\
 N(\rr(a^*))&=&\sum\limits_{i=1}\limits^n \big(\langle a^*, a_i \rangle N(b_i)\big).
 \end{eqnarray*}
 So Eq.\,(\ref{eq:rr-2}) $\Leftrightarrow$ Eq.\,(\ref{eq:cybe-1}). This completes the proof.
 \end{proof}

\subsection{Relative Rota-Baxter operators on Nijenhuis pre-Lie algebras}\label{se:o}
The operator characterization in Theorem \ref{thm:rr} leads naturally to relative Rota-Baxter operators on Nijenhuis pre-Lie algebras and to further bialgebra constructions.

 \begin{defi}\label{defi:4d1} Let $((A, \ci), N)$ be a Nijenhuis pre-Lie algebra, $(V, \rho, \vp)$ a representation of $(A, \ci)$ and $\a:V\lr V$ a linear map. A linear map $T:V\lr A$ is called a {\bf weak relative Rota-Baxter-operator associated to $(V, \rho, \vp)$ and $\a$} if $T$ satisfies
 \begin{eqnarray}
 &T(u)\ci T(v)=T\Big(\rho\big(T(u)\big)v+\vp\big(T(v)\big)u\Big),\forall~u, v\in V, &\label{eq:z4}\\
 &N T=T\a.&\label{eq:z5}
 \end{eqnarray}
 Moreover, if $((V, \rho, \vp), \a)$ is a representation of $((A, \ci), N)$, then $T$ is called a {\bf relative Rota-Baxter-operator associated to $(V, \rho,$ $\vp, \a)$}.
 \end{defi}

 \begin{ex}\label{ex:ex100} Let $((A, \ci), N)$ be a Nijenhuis pre-Lie algebra. Then the identity map $\id$ on $(A, \ci)$ is a relative Rota-Baxter-operator associated to $((A, L, 0), N)$ or $((A, 0, R), N)$.
 \end{ex}

 Theorem \ref{thm:rr} can be restated in terms of relative Rota-Baxter operators.

 \begin{cor}\label{cor:4c3} Let $((A, \ci), N)$ be a Nijenhuis pre-Lie algebra, $r\in A\o A$ symmetric and $S:A\lr A$ a linear map. Then $r$ is a solution of the $S$-Nijenhuis \sss-equation in $((A, \ci), N)$ if and only if $r^{\sharp}$ is a weak relative Rota-Baxter-operator associated to $(A^*, L^*-R^*, -R^* )$ and $S^*$. Moreover, if $((A, \ci), N)$ is an $S$-admissible Nijenhuis pre-Lie algebra, then $r$ is a solution of the $S$-Nijenhuis \sss-equation in $((A, \ci), N)$ if and only if $r^{\sharp}$ is a relative Rota-Baxter-operator associated to the representation $((A^*, L^*-R^*, -R^*), S^*)$.
 \end{cor}

We next show how relative Rota-Baxter operators yield solutions to the $S$-Nijenhuis \sss-equation in semidirect-product Nijenhuis pre-Lie algebras.

 \begin{thm}\label{thm:ss1} Let $((A, \ci), N)$ be a Nijenhuis pre-Lie algebra, $(V, \rho, \vp)$ a representation of $(A, \ci)$, $S:A\lr A$ and $\a, \beta:V\lr V$ linear maps. Then the following conditions are equivalent:
 \begin{enumerate}[(1)]
 \item \label{it:ss1} There is a Nijenhuis pre-Lie algebra $(A\ltimes_{\rho, \vp} V, N+\a)$ such that the linear map $S+\beta$ on $A\oplus V$ is admissible to $(A\ltimes_{\rho, \vp} V, N+\a)$.
 \item \label{it:ss2} There is a Nijenhuis pre-Lie algebra $(A\ltimes_{\rho^*-\vp^*, -\vp^* } V^*, N+\beta^*)$ such that the linear map $S+\a^*$ on $A\oplus V^*$ is admissible to $(A\ltimes_{\rho^*-\vp^*, -\vp^* } V^*, N+\beta^*)$.
 \item \label{it:ss3} The following conditions are satisfied:
 \begin{enumerate}[(a)]
 \item \label{it:as1} $((V, \rho, \vp), \a)$ is a representation of $((A, \ci), N)$;
 \item \label{it:as2} $S$ is admissible to $((A, \ci), N)$;
 \item \label{it:as3} $\beta$ is admissible to $((A, \ci), N)$ on $(V, \rho, \vp)$;
 \item \label{it:as4} For all $x\in A$ and $v\in V$, the following equations hold:
 \begin{eqnarray}
 &&\beta\big(\rho(x)\a(v)\big)+\rho\big(S^2(x)\big)v=\rho\big(S(x)\big)\a(v)
 +\beta\big(\rho(S(x))v\big),\label{eq:hj} \\
 &&\beta\big(\vp(x)\a(v)\big)+\vp\big(S^2(x)\big)v=\vp\big(S(x)\big)\a(v)
 +\beta\big(\vp(S(x))v\big).\label{eq:hj2}
 \end{eqnarray}
 \end{enumerate}
 \end{enumerate}
 \end{thm}

 \begin{proof} $(\ref{it:ss1}\Leftrightarrow\ref{it:ss3}):$ By Proposition \ref{pro:repnplie}, $(A\ltimes_{\rho, \vp} V, N+\a)$ is a Nijenhuis pre-Lie algebra if and only if $((V, \rho, \vp), \a)$ is a representation of $((A, \ci), N)$. To express the condition that $S+\beta$ on $A\oplus V$ is admissible to $(A\ltimes_{\rho, \vp} V, N+\a)$ in explicit formulas, we perform the following calculations. For $x, y\in A$ and $u, v\in V$, we have
 \begin{align*}
 (S+\beta)((N+\alpha)(x+u)\ci_{\heartsuit} (y+v))&=S(N(x)\ci y)+\beta(\rho(N(x))v)+\beta(\vp(y)\alpha(u)),\\
 ((x+u)\ci_{\heartsuit} (S+\beta)^2(y+v))&=x\ci S^2(y)+\rho(x)\beta^2(v)+\vp(S^2(y))u,\\
 (N+\alpha)(x+u)\ci_{\heartsuit} (S+\beta)(y+v)&=N(x)\ci S(y)+\rho(N(x))\beta(v)+\vp(S(y))\alpha(u),\\
 (S+\beta)((x+u)\ci_{\heartsuit} (S+\beta)(y+v))&=S(x\ci S(y))+\beta(\rho(x)\beta(v))+\beta(\vp(S(y))u),\\
 (S+\beta)((x+u)\ci_{\heartsuit} (N+\alpha)(y+v))&=S(x\ci N(y))+\beta(\rho(x)\alpha(v))+\beta(\vp(N(y))u),\\
 (S+\beta)^2(x+u)\ci_{\heartsuit} (y+v)&=S^2(x)\ci y+\rho(S^2(x))v+\vp(y)\beta^2(u),\\
 (S+\beta)(x+u)\ci_{\heartsuit} (N+\alpha)(y+v)&=S(x)\ci N(y)+\rho(S(x))\alpha(v)+\vp(N(y))\beta(u),\\
 (S+\beta)((S+\beta)(x+u)\ci_{\heartsuit} (y+v))&=S(S(x)\ci y)+\beta(\rho(S(x))v)+\beta(\vp(y)\beta(u)).
 \end{align*}
 Therefore Eq.\,(\ref{eq:dual3-1}) holds (where $N$ is replaced by $N+\a$, $S$ by $S+\beta$, $x$ by $x+u$, $y$ by $y+v$) if and only if Eq.\,(\ref{eq:dual3-1}) (corresponding to $u=v=0$), Eq.\,(\ref{eq:dual3}) (corresponding to $y=u=0$) and Eq.\,(\ref{eq:hj2}) hold, where $x$ is replaced by $y$, $v$ by $u$ (corresponding to $x=v=0$). Eq.\,(\ref{eq:dual2-1}) holds (where $N$ is replaced by $N+\a$, $S$ by $S+\beta$, $x$ by $x+u$, $y$ by $y+v$) if and only if Eq.\,(\ref{eq:dual2-1}) (corresponding to $u=v=0$), Eq.\,(\ref{eq:dual2}) (where $x$ is replaced by $y$, $v$ by $u$ (corresponding to $x=v=0$)) and Eq.\,(\ref{eq:hj}) hold (corresponding to $y=u=0$). Hence Item \ref{it:ss1} $\Leftrightarrow$ Item \ref{it:ss3}.

 $(\ref{it:ss2}\Leftrightarrow\ref{it:ss3}):$  Based on $\ref{it:ss1}\Leftrightarrow\ref{it:ss3}$, replace $V$ by $V^*$, $\rho$ by $\rho^*-\vp^*$, $\vp$ by $-\vp^*$, $\beta$ by $\a^*$, $\a$ by $\beta^*$. Then we have \ref{it:ss2} holds if and only if
 \begin{enumerate}[(i)]
 \item \label{it:i1} $\big(V^*, \rho^*-\vp^*, -\vp^* , \beta^*\big)$ is a representation of $((A, \ci), N)$, that is, (\ref{it:ss3}\ref{it:as3}) holds;
 \item \label{it:i2} $S$ is admissible to $((A, \ci), N)$, that is, (\ref{it:ss3}\ref{it:as2}) holds;
 \item \label{it:i3} $\a^*$ is admissible to $(A,  N)$ on $\big(V^*, \rho^*-\vp^*, -\vp^* \big)$, that is, (\ref{it:ss3}\ref{it:as1}) holds;
 \item \label{it:i4} For all $x\in A$ and $v^*\in V^*$, the following equations hold:
 \begin{eqnarray}
 &&\hspace{-11mm}\a^*\big((\rho^*-\vp^*)(x)\b^*(v^*)\big)+(\rho^*-\vp^*)\big(S^2(x)\big)v^*=(\rho^*-\vp^*)\big(S(x)\big)\b^*(v^*)
 +\a^*\big((\rho^*-\vp^*)(S(x))v^*\big),\label{eq:vb1}\\
 &&\a^*\big((-\vp^*)(x)\b^*(v^*)\big)+(-\vp^*)\big(S^2(x)\big)v^*=(-\vp^*)\big(S(x)\big)\b^*(v^*)
 +\a^*\big((-\vp^*)(S(x))v^*\big).\label{eq:vb2}
 \end{eqnarray}
 \end{enumerate}

 \noindent Then Eq.\,(\ref{eq:vb2}) holds if and only if Eq.\,(\ref{eq:hj2}) holds. Thus by Eq.\,(\ref{eq:vb2}), we obtain that Eq.\,(\ref{eq:vb1}) holds if and only if Eq.\,(\ref{eq:hj}) holds. Therefore Item \ref{it:ss2} holds if and only if Item \ref{it:ss3} holds.
 \end{proof}

We now use these relative Rota-Baxter operators to construct Nijenhuis pre-Lie bialgebras.

 \begin{thm}\label{thm:las2} Let $((A,\circ),N)$ be $\b$-admissible with respect to $(V, \rho, \vp)$, $S:A\lr A$, $\a:V\lr V$ be linear maps. Let $T:V\lr A$ be a linear map which is identified as an element in $V^*\otimes A\subseteq (A\ltimes_{\rho^*-\vp^*, -\vp^* } V^*)\o (A\ltimes_{\rho^*-\vp^*, -\vp^* } V^*)$.
 \begin{enumerate}[(1)]
 \item \label{it:ufe} The element $r=T+\tau(T)$ is a symmetric solution of the $(S+\a^*)$-Nijenhuis \sss-equation in the Nijenhuis pre-Lie algebra $(A\ltimes_{\rho^*-\vp^*, -\vp^* }V^*, N+\beta^*)$ if and only if $T$ is a weak relative Rota-Baxter-operator associated to $(V, \rho, \vp)$ and $\a$, and $T \beta=S T$.
 \item \label{it:ufp} Assume that $((V, \rho, \vp), \a)$ is a representation of $((A, \ci), N)$. If $T$ is a relative Rota-Baxter-operator associated to $((V, \rho, \vp), \a)$ and $T\circ \beta=S\circ T$, then $r=T+\tau(T)$ is a symmetric solution of the $(S+\a^*)$-Nijenhuis \sss-equation in the Nijenhuis pre-Lie algebra $(A\ltimes_{\rho^*-\vp^*, -\vp^*}V^*, N+\beta^*)$. Moreover, if $(A,N)$ is $S$-admisssible and Eqs.(\ref{eq:hj}) and (\ref{eq:hj2}) hold, then $(A\ltimes_{\rho^*-\vp^*, -\vp^*}V^*, N+\beta^*)$ is $(S+\a^*)$-admissible. In this case, there is a Nijenhuis pre-Lie bialgebra $((A\ltimes_{\rho^*-\vp^*, -\vp^*}V^*, N+\beta^*), \delta_{r}, S+\a^*)$, where $\delta_r$ is defined by Eq.\,(\ref{eq:cop}) with $r=T+\tau(T)$.
 \end{enumerate}
 \end{thm}

\begin{proof} \ref{it:ufe} Let $\{e_1,e_2,...,e_n\}$ be a basis of $V$ and $\{e^1,e^2,...,e^n\}$ be its dual basis. Then $r=T+\tau(T)$ corresponds to $\sum\limits_{i=1}\limits^n \Big(T(e_i)\o e^i+e^i\o T(e_i)\Big) \in(A\ltimes_{(\rho^*-\vp^*, -\vp^*}V^*)\o (A\ltimes_{(\rho^*-\vp^*, -\vp^*}V^*)$. By Theorem \ref{thm:rr}, we find that $r=T+\tau(T)$ is a symmetric solution of the \sss-equation in $A\ltimes_{\rho^*-\vp^*, -\vp^* }V^*$ if and only if Eq.\,(\ref{eq:rr-1}) holds for the regular representation $(A\ltimes_{\rho^*-\vp^*, -\vp^* }V^*, \lll, \rrr)$ of $A\ltimes_{\rho^*-\vp^*, -\vp^* }V^*$.
More precisely,
For all $a^*, b^*\in A^*$ and $u, v\in V$, we have
\begin{eqnarray*}
 &&\rr(a^*+u)\ci_{\heartsuit} \rr(b^*+v)=\rr\Big((\lll^*-\rrr^*)(\rr(a^*+u))(b^*+v)-\rrr^*(\rr(b^*+v))(a^*+u)\Big)
\end{eqnarray*}
if and only if
$$T(u)\ci T(v)=T\Big(\rho\big(T(u)\big)v+\vp\big(T(v)\big)u\Big).$$

Furthermore, note that
\begin{eqnarray*}
 \big((N+\beta^*)\o \id\big)(r)&=&\sum\limits_{i=1}\limits^n \Big(N (T(e_i))\o e^i+\beta^*(e^i)\o T(e_i)\Big), \\
 \big(\id\o (S+\a^*)\big)(r)&=&\sum\limits_{i=1}\limits^n \Big(T(e_i)\o\a^*(e^i)+e^i\o S (T(e_i))\Big).
 \end{eqnarray*}
On the other hand,
\begin{eqnarray*}
 \sum_{i=1}^n\beta^*(e^i)\o T(e_i)
 &=&\sum_{i=1}^n\sum_{j=1}^n \big\langle\beta^*(e^i), e_j\big\rangle e^j\o T(e_i)=\sum\limits_{j=1}\limits^n e^j\o \sum\limits_{i=1}\limits^n \big\langle e^i, \beta(e_j)\big\rangle T(e_i)\\
 &=&\sum\limits_{i=1}\limits^n e^i\o \sum\limits_{j=1}\limits^n \big\langle e^j, \beta(e_i)\big\rangle T(e_j)=\sum\limits_{i=1}\limits^n e^i\o T(\beta(e_i)),
 \end{eqnarray*}

\begin{eqnarray*}
 \sum\limits_{i=1}\limits^n T(e_i)\o\a^*(e^i)
 &=&\sum\limits_{i=1}\limits^n\sum\limits_{j=1}\limits^n T(e_i)\o \big\langle\a^*(e^i), e_j\big\rangle e^j\\
 &=&\sum\limits_{j=1}\limits^n \Big(\sum\limits_{i=1}\limits^n T(e_i) \big\langle e^i, \a(e_j) \big\rangle\Big)\o e^j=\sum\limits_{i=1}\limits^n T(\a(e_i))\o e^i.
 \end{eqnarray*}
Therefore $\big((N+\beta^*)\o \id\big)(r)=\big(\id\o (S+\a^*)\big)(r)$ if and only if $T\beta=S T$ and $N T=T\a$. Hence the conclusion follows.

 \ref{it:ufp} It follows Item \ref{it:ufe} and Theorem \ref{thm:ss1}.
 \end{proof}

Now in the general discussion in Theorem~\ref{thm:ss1}, let us take $\beta=\pm\a$ or $-\a+\theta\id$ or $\beta=\theta\a^{-1}$. These cases all have the property that the double dual of a representation is itself. To unify the notions, let $\Pi$ be one of the involutive Laurent transformations: $\Pi(t)=\pm t$, or $\Pi(t)=-t+\theta$, or $\Pi(t)=\theta t^{-1}$ when $t$ is invertible and $0\neq \theta\in K$. Thus $\Pi^2=\Pi\ci \Pi=\id$. Then the special cases above can be denoted by $\Pi(\a)$.

 To emphasize, for each $\Pi$ in the set
 \begin{eqnarray*}
 \{\pm t\}\cup(-t+K^{\times})\cup K^{\times}t^{-1},  \text{ where } K^{\times}:=K\setminus\{0\},
 \end{eqnarray*}
 we have $\Pi^2(\a)=\a$ and $\Pi(\a^*)={\Pi(\a)}^*$. Moreover, for any linear map $T:V\lr A$, it is obvious that $T\Pi(\a)=\Pi(N)T$ when $T\a=NT$.

 Applying Theorem \ref{thm:ss1}, we have

\begin{pro}\label{pro:vv1} Let $((A, \ci), N)$ be a Nijenhuis pre-Lie algebra, and $((V, \rho, \vp), \a)$ a representation of $((A, \ci), N)$. For $\Pi\in \{\pm t\}\cup(-t+K^{\times})\cup K^{\times}t^{-1}$, there is a Nijenhuis pre-Lie algebra $\big(A\ltimes_{\rho^*-\vp^*,-\vp^*}V^*, N+{\Pi(\a)}^*\big)$ such that $(\Pi(N)+\a^*)$-admissible if and only if the $\Pi$-admissible equations \big(associated to the quadruple $((V, \rho, \vp), \a)$\big) hold. In concrete terms, we have the following conclusions.
 \begin{enumerate}[(1)]
 \item \label{it:vv11}When $\Pi=\theta t$ with $\theta\in\{\pm1\}$, that is, $\beta=\theta\a$, $S=\theta N$, the $\Pi$-admissible equations are
 \begin{eqnarray}
 &x\ci N^2(y)+\theta N^2(x\ci y)=(1+\theta)N(x\ci N(y)),&\label{eq:kk3}\\
 &N^2(x)\ci y+\theta N^2(x\ci y)=(1+\theta)N(N(x)\ci y),&\label{eq:kk4}\\
 &\rho(x)\a^2(v)+\theta\a^2\big(\rho(x)v\big)=(1+\theta)\a\big(\rho(x)\a(v)\big),&\label{eq:kk5}\\
 &\vp(x)\a^2(v)+\theta\a^2\big(\vp(x)v\big)=(1+\theta)\a\big(\vp(x)\a(v)\big),&\label{eq:kk6}\\
 &\rho\big(N^2(x)\big)v+\theta\a^2\big(\rho(x)v\big)=(1+\theta)\a\Big(\rho\big(N(x)\big)v\Big),&\label{eq:kk7}\\
 &\vp\big(N^2(x)\big)v+\theta\a^2\big(\vp(x)v\big)=(1+\theta)\a\Big(\vp\big(N(x)\big)v\Big),&\label{eq:kk8}
 \end{eqnarray}
 where $x, y\in A, v\in V$.
 \item \label{it:vv22} When $\Pi=-t+\theta$ with $\theta\neq 0$, that is, $\beta=-\a+\theta\id_V$, $S=-N+\theta\id_A$, the $\Pi$-admissible equations are 
 \begin{eqnarray}
 &x\ci N^2(y)+\theta N(x\ci y)=N^2(x\ci y)+\theta x\ci N(y),&\label{eq:kkb1}\\
 &N^2(x)\ci y+\theta N(x\ci y)=N^2(x\ci y)+\theta N(x)\ci y,&\label{eq:kkb2}\\
 &\rho(x)\a^2(v)+\theta\a\big(\rho(x)v\big)=\a^2\big(\rho(x)v\big)+\theta\rho(x)\a(v),&\label{eq:kkb3}\\
 &\vp(x)\a^2(v)+\theta\a\big(\vp(x)v\big)=\a^2\big(\vp(x)v\big)+\theta\vp(x)\a(v),&\label{eq:kkb4}\\
 &\rho\big(N^2(x)\big)v+\theta\a\big(\rho(x)v\big)=\a^2\big(\rho(x)v\big)
 +\theta\rho\big(N(x)\big)v,&\label{eq:kkb5}\\
 &\vp\big(N^2(x)\big)v+\theta\a\big(\vp(x)v\big)=\a^2\big(\vp(x)v\big)
 +\theta\vp\big(N(x)\big)v,&\label{eq:kkb6}
 \end{eqnarray}
 where $x, y\in A, v\in V$.
 \item \label{it:vv33} When $\Pi=\theta t^{-1}$, $\theta\neq 0$ (in which case assume that $N$ and $\a$ is invertible), that is, $\beta=\theta\a^{-1}$, $S=\theta N^{-1}$, the $\Pi$-admissible equations are 
 \begin{eqnarray}
 &N\big(\theta x\ci y+x\ci N^2(y)\big)=(\theta\id_A+N^2)\big(x\ci N(y)\big),&\label{eq:kkc1}\\
 &N\big(\theta x\ci y+N^2(x)\ci y\big)=(\theta\id_A+N^2)\big(N(x)\ci y\big),&\label{eq:kkc2}\\
 &\a\Big(\theta \rho(x)v +\rho\big(N^2(x)\big)v\Big)
 =(\theta\id_V+\a^2)\Big(\rho\big(N(x)\big)v\Big),&\label{eq:kkc3}\\
 &\a\Big(\theta \vp(x)v +\vp\big(N^2(x)\big)v\Big)
 =(\theta\id_V+\a^2)\Big(\vp\big(N(x)\big)v\Big),&\label{eq:kkc4}\\
 &\a\Big(\theta \rho(x)v +\rho(x)\a^2(v)\Big)
 =(\theta\id_V+\a^2)\big(\rho(x)\a(v)\big),&\label{eq:kkc5}\\
 &\a\Big(\theta \vp(x)v +\vp(x)\a^2(v)\Big)
 =(\theta\id_V+\a^2)\big(\vp(x)\a(v)\big),&\label{eq:kkc6}
 \end{eqnarray}
 where $x, y\in A, v\in V$.
 \end{enumerate}
 \end{pro}

\begin{proof}
The statements are consequences of Theorem \ref{thm:ss1} in the case that $\beta=\Pi(\a)$ and $S=\Pi(N)$. A detailed proof follows the one for \cite[Proposition 4.22]{BGM}.
\end{proof}

\begin{thm}\label{thm:last1}Let $((A, \ci), N)$ be a Nijenhuis pre-Lie algebra, $((V, \rho, \vp), \a)$ be a representation of $((A, \ci), N)$. Let $\Pi\in \{\pm t\}\cup(-t+K^{\times})\cup K^{\times}t^{-1}$.
 \begin{enumerate}[(1)]
 \item \label{it:last1} Let $((V, \rho, \vp), \Pi(\a))$ be an admissible quadruple of $((A, \ci), N)$. Then $r=T+\tau(T)$ is a symmetric solution of the $(\Pi(N)+\a^*)$-Nijenhuis \sss-equation in the Nijenhuis pre-Lie algebra $\big(A\ltimes_{\rho^*-\vp^*, -\vp^* }V^*, N+\Pi(\a^*)\big)$ if and only if $T$ is a relative Rota-Baxter-operator associated to $((V, \rho, \vp), \a)$.
 \item \label{it:last2} Assume the validity of the $\Pi$-admissible equations, given respectively by Eqs.\,(\ref{eq:kk3})--(\ref{eq:kk8}) for $\Pi=\pm t$, by Eqs.\,(\ref{eq:kkb1})--(\ref{eq:kkb6}) for $\Pi\in-t+K^{\times}$,  and by Eqs.\,(\ref{eq:kkc1})--(\ref{eq:kkc6}) for $\Pi\in K^{\times}t^{-1}$. Then there is a $(\Pi(N)+\a^*)$-admissible Nijenhuis pre-Lie algebra $\big(A\ltimes_{\rho^*-\vp^*, -\vp^* }V^*, N+\Pi(\a^*)\big)$. If $T$ is a  relative Rota-Baxter-operator associated to $((V, \rho, \vp), \a)$, then $r=T+\tau(T)$ is a symmetric solution of the $(\Pi(N)+\a^*)$-Nijenhuis \sss-equation in the Nijenhuis pre-Lie algebra $\big(A\ltimes_{\rho^*-\vp^*, -\vp^* }V^*, N+\Pi(\a^*)\big)$. Furthermore there is a Nijenhuis pre-Lie bialgebra $\big((A\ltimes_{\rho^*-\vp^*, -\vp^* }V^*, N+\Pi(\a^*)), \delta, \Pi(N)+\a^*\big)$, where the linear map $\delta=\delta_r$ is defined by Eq.\,\eqref{eq:cop} with $r=T+\tau(T)$.
 \end{enumerate}
 \end{thm}

 \begin{proof}
 \ref{it:last1} follows from Theorem \ref{thm:las2}\,\ref{it:ufe}.

 \ref{it:last2} follows from Proposition \ref{pro:vv1} and Theorem \ref{thm:las2}\,\ref{it:ufp}.
 \end{proof}

These constructions complete the Nijenhuis pre-Lie bialgebra framework. We next ask when the underlying pre-Lie bialgebra also determines a compatible Lie bialgebra.

\section{Balanced Nijenhuis pre-Lie bialgebras}\label{se:balance}
The balanced condition provides the bridge from pre-Lie bialgebras to Lie bialgebras and, in the presence of Nijenhuis operators, to Nijenhuis Lie bialgebras. We begin by recalling the Lie bialgebra definitions from \cite{Dr}. A {\bf Lie coalgebra} is a pair $(A, \d)$, where $A$ is a vector space and $\d:A\lr A\o A$ (we use the sumless Sweedler notation \cite{Sw}: $\d(x):=x_{[1]}\o x_{[2]}$ ($=\sum x_{[1]}\o x_{[2]}$)) is a linear map such that for all $x\in A$,
 \begin{eqnarray*}
 &x_{[1]}\o x_{[2]}=-x_{[2]}\o x_{[1]},&\\
 &x_{[1]}\o x_{[2][1]}\o x_{[2][2]}+x_{[2][1]}\o x_{[2][2]}\o x_{[1]}+ x_{[2][2]}\o x_{[1]}\o x_{[2][1]}=0.&
 \end{eqnarray*}

 A {\bf Lie bialgebra} is a triple $(A, [,], \d)$, where the pair $(A, [,])$ is a Lie algebra, and the pair $(A, \d)$ is a Lie coalgebra, such that, for all $x, y\in A$,
 \begin{eqnarray}
 &\d([x, y])=[x, y_{[1]}]\o y_{[2]}+ y_{[1]}\o [x, y_{[2]}]-[y, x_{[1]}]\o x_{[2]}-x_{[1]}\o [y, x_{[2]}].& \label{eq:liebialg}
 \end{eqnarray}

 \begin{defi}\label{de:balance} A pre-Lie bialgebra $(A, \ci, \D)$ is called {\bf balanced} if
 \begin{eqnarray}
 &x_{(1)}\ci y\o x_{(2)}+y_{(2)}\o y_{(1)}\ci x=y_{(1)}\ci x\o y_{(2)}+x_{(2)}\o x_{(1)}\ci y,\quad x, y\in A.&\label{eq:balance}
 \end{eqnarray}
 \end{defi}

 \begin{ex} \label{ex:balance-1}
 All the pre-Lie bialgebras in Example \ref{ex:nliebialg} are balanced.
 \end{ex}

 \begin{pro}\label{pro:balance-a} Let $(A, \ci, \D)$ be a pre-Lie bialgebra that is commutative $($$x\ci y=y\ci x$ for all $x, y\in A$$)$ and cocommutative $($$x_{(1)}\o x_{(2)}=x_{(2)}\o x_{(1)}$ for all $x\in A$$)$. Then $(A, \ci, \D)$ is balanced.
 \end{pro}

 \begin{proof} By Eq.\,(\ref{eq:b1}), we obtain $x\ci y_{(1)}\o y_{(2)}=y\ci x_{(1)}\o x_{(2)}$ since $(A, \ci, \D)$ is commutative and cocommutative. Then Eq.\,(\ref{eq:balance}) holds automatically, i.e., $(A, \ci, \D)$ is balanced.
 \end{proof}

 \begin{lem}\label{lem:balance} Let $(A, \circ, \D)$ be a pre-Lie bialgebra. Define
 \begin{eqnarray}
 &[x, y]:=x\circ y-y\circ x,\quad \d(x):=x_{[1]}\o x_{[2]}:=x_{(1)}\o x_{(2)}-x_{(2)}\o x_{(1)}, \quad x, y\in A,&\label{eq:comm}
 \end{eqnarray}
 then $(A, [,], \d)$ is a Lie bialgebra if and only if $(A, \circ, \D)$ is balanced.
 \end{lem}

 \begin{proof} First we know that $(A, [,])$ is a Lie algebra. Dually, $(A, \d)$ is a Lie coalgebra. For all $x, y\in A$, by Eq.\,(\ref{eq:comm}), we calculate as follows.
 \begin{eqnarray*}
 \hbox{LHS of~}\text{Eq}.\,(\ref{eq:liebialg})
 \hspace{-3mm}&\stackrel{}{=}&\hspace{-3mm}(x\circ y)_{(1)}\o (x\circ y)_{(2)}-(x\circ y)_{(2)}\o (x\circ y)_{(1)}\\
 &&\hspace{-3mm}-(y\circ x)_{(1)}\o (y\circ x)_{(2)}+(y\circ x)_{(2)}\o (y\circ x)_{(1)}\\
 \hbox{RHS of~Eq.}\,(\ref{eq:liebialg})
 \hspace{-3mm}&\stackrel{}{=}&\hspace{-3mm}x\circ y_{(1)}\o y_{(2)}-y_{(1)}\circ x\o y_{(2)}-x\circ y_{(2)}\o y_{(1)}+y_{(2)}\circ x\o y_{(1)}\\
 \hspace{-3mm}&&\hspace{-3mm}+y_{(1)}\o x\circ y_{(2)}-y_{(1)}\o y_{(2)}\circ x-y_{(2)}\o x\circ y_{(1)}+y_{(2)}\o y_{(1)}\circ x\\
 \hspace{-3mm}&&\hspace{-3mm}-y\circ x_{(1)}\o x_{(2)}+x_{(1)}\circ y\o x_{(2)}+y\circ x_{(2)}\o x_{(1)}-x_{(2)}\circ y\o x_{(1)}\\
 \hspace{-3mm}&&\hspace{-3mm}-x_{(1)}\o y\circ x_{(2)}+x_{(1)}\o x_{(2)}\circ y+x_{(2)}\o y\circ x_{(1)}-x_{(2)}\o x_{(1)}\circ y.
 \end{eqnarray*}
 Then by Eq.\,(\ref{eq:b1}) or  Eq.\,(\ref{eq:b2}), we find that Eq.\,(\ref{eq:liebialg}) holds for $[,]$ and $\d$ given in Eq.\,(\ref{eq:comm}) if and only if Eq.\,(\ref{eq:balance}) holds, i.e., $(A, \circ, \D)$ is balanced.
 \end{proof}


Let us further recall from \cite{LM} the notion of Nijenhuis Lie bialgebra. A {\bf Nijenhuis Lie bialgebra} is a quintuple $(A, [,], \d, T, P)$, where
  \begin{enumerate}
    \item \label{it:liebialg1} $((A, [,]), T)$ is Nijenhuis Lie algebra;
    \item \label{it:liebialg2} $((A, \d), P)$ is Nijenhuis Lie coalgebra;
    \item \label{it:liebialg3} $(A, [,], \d)$ is Lie bialgebra;
    \item \label{it:liebialg4} For all $x, y\in A$, the equations below hold:
 \begin{eqnarray}
 &P([T(x), y])+[x, P^2(y)]=[T(x), P(y)]+P([x, P(y)]),&\label{eq:nliebialg1-1-1}\\
 &P(T(x)_{[1]})\o T(x)_{[2]}+x_{[1]}\o T^2(x_{[2]})=P(x_{[1]})\o T(x_{[2]})+T(x)_{[1]}\o T(T(x)_{[2]}).&\label{eq:nliebialg2-2-2}
 \end{eqnarray}
 \end{enumerate}

 The next theorem makes precise when a Nijenhuis pre-Lie bialgebra induces a Nijenhuis Lie bialgebra.

 \begin{thm}\label{thm:balance-1} Let $(A, \circ, \D, N, S)$ be a Nijenhuis pre-Lie bialgebra. Then $(A, [-,-], \d, N, S)$, where $[-,-]$ and $\d$ are given by Eq.\,\eqref{eq:comm}, is a Nijenhuis Lie bialgebra if and only if $(A, \circ, \D)$ is balanced.
 \end{thm}

 \begin{proof} For all $x, y\in A$, one calculates
 \begin{eqnarray*}
 &&\hspace{-13mm}S([N(x), y])+[x, S^2(y)]-[N(x), S(y)]-S([x, S(y)])\\
 &&\hspace{-3mm}\stackrel{(\ref{eq:comm})}{=}\hspace{-1mm}S(N(x)\ci y)-S(y\ci N(x))+x\ci S^2(y)-S^2(y)\ci x-N(x)\ci S(y)\\
 &&\qquad+S(y)\ci N(x)-S(x\ci S(y))+S(S(y)\ci x)\\
 &&\hspace{-5.5mm}\stackrel{(\ref{eq:dual3-1})(\ref{eq:dual2-1})}{=}0.
 \end{eqnarray*}
 So Eq.\,(\ref{eq:nliebialg1-1-1}) holds. Similarly, we obtain Eq.\,(\ref{eq:nliebialg2-2-2}) by Eqs.\,(\ref{eq:nliebialg1}) and (\ref{eq:nliebialg2}). Furthermore, $N$ (resp. $S$) is a Nijenhuis operator on the Lie algebra $(A, [-,-])$ (resp. Lie coalgebra $(A, \d)$) since $N$ is a Nijenhuis operator on the pre-Lie algebra $(A, \ci)$ (resp. pre-Lie coalgebra $(A, \D)$).  So applying Lemma \ref{lem:balance} completes the proof.
 \end{proof}

By Example \ref{ex:balance-1} and Theorem \ref{thm:balance-1}, we get the following Nijenhuis Lie bialgebras.

\begin{ex}\label{ex:balance-2} Let $k_1, k_2, k_3, \ell_1, \ell_2, \ell_3$ be parameters. Assume that $(A, [-,-])$ is the Lie algebra induced by the pre-Lie algebra given in Example \ref{ex:sym}. Then $[-,-]$ is given by
 \begin{center}
   \begin{tabular}{r|cc}
          $[-,-]$ & $e$  & $f$  \\
          \hline
           $e$ & $0$  & $e$  \\
           $f$ & $-e$  &  $0$ \\
        \end{tabular}
   \end{center}
 \begin{enumerate}[(I)]
    \item Set
   \begin{center}$ \left\{
        \begin{array}{l}
        \d(e)=0,\\
        \d(f)=k_2 (f\o e-e\o f).
        \end{array}\right. $
             \end{center}
        Then $(A, \d)$ is a Lie coalgebra and further $(A, [,], \d)$ is a Lie bialgebra. Define the linear maps $N, S: A\lr  A$ by:
      \begin{center} $\left\{
        \begin{array}{l}
         N(e)=k_2 \ell_3 e,\\
         N(f)=(k_1 \ell_3+k_2 \ell_2)e+k_2 \ell_3 f.\\
         \end{array}
           \right.$\quad $S=N$.
             \end{center}
 Then by Theorem \ref{thm:balance-1} we see that $(A, [,], \d, N, S)$ is a \N Lie bialgebra.
    \item Set
   \begin{center}$ \left\{
        \begin{array}{l}
        \d(e)=k_3 (f\o e-e\o f),\\
        \d(f)=0
        \end{array}\right. $
             \end{center}
        Then $(A, \d)$ is a Lie coalgebra and further $(A, [,], \d)$ is a Lie bialgebra. Define the linear maps $N, S: A\lr  A$ by:
      \begin{center} $\left\{
        \begin{array}{l}
         N(e)=0\\
         N(f)=k_3 \ell_2 f\\
         \end{array}
           \right.$,\quad $\left\{
        \begin{array}{l}
         S(e)=k_3 \ell_2 e\\
         S(f)=k_3 \ell_2 f\\
         \end{array}
           \right.$
             \end{center}
 Then $(A, [,], \d, N, S)$ is a \N Lie bialgebra.
 \end{enumerate}
 \end{ex}

 The preceding section explains when the pre-Lie structure descends to a compatible Nijenhuis Lie bialgebra. We now return to the pre-Lie setting and consider a complementary question: how can a Nijenhuis pre-Lie bialgebra be combined with compatible module and comodule data to produce a larger bialgebra?

\section{Crossed modules of Nijenhuis pre-Lie bialgebras}\label{se:bpre}
The theory developed so far concerns Nijenhuis operators and compatibility conditions on a fixed pre-Lie bialgebra. To investigate extensions and new constructions, we also need to understand how such a bialgebra interacts with an auxiliary vector space. A representation describes the action of the pre-Lie product, whereas a corepresentation describes the corresponding coalgebraic data. When Nijenhuis operators are present on both sides, these structures must be compatible not only with the product and coproduct, but also with the additional operators.

This motivates a crossed-module framework in which actions and coactions are considered simultaneously. It also leads to a pre-Lie counterpart of the braided-bialgebra and biproduct viewpoint familiar from Lie-theoretic constructions \cite{Maj}. The central problem is to determine when data on a bialgebra $A$ and an auxiliary space $V$ define a Nijenhuis pre-Lie bialgebra on the direct sum $A\oplus V$. We first construct crossed modules from triangular structures and then derive a criterion for the corresponding biproduct. For a Nijenhuis pre-Lie coalgebra $((V, \nabla), \b)$, we use the notation $\nabla(u):=u^{(1)}\o u^{(2)}$ for $u\in V$.

 \subsection{Nijenhuis pre-Lie crossed modules}
The first step is to formulate the mixed compatibility relations between a representation and a corepresentation. They will later account for the interaction between the two summands in the biproduct construction, while the Nijenhuis conditions ensure compatibility with the operators on both sides.

 \begin{defi}\label{de:yd} Let $((A, \ci, \D), N, S)$ be a Nijenhuis pre-Lie bialgebra. A {\bf crossed module of $(A, \ci,$ $\D)$} is a quintuple $(V, \rho, \vp, \xi, \eta)$, where $(V, \rho, \vp)$ is a representation of $(A, \ci)$, $(V, \xi, \eta)$ is a corepresentation of $(A, \D)$ such that for all $x\in A, v\in V$, the following conditions hold:
 \begin{eqnarray}
 &\xi\big((\rho-\vp)(x)v\big)=x\ci v_{-1}\o v_0+v_{-1}\o (\rho-\vp)(x)v_0+x_{(1)}\o (\rho-\vp)(x_{(2)})v,& \label{eq:yd1}\\
 &\eta\big((\rho-\vp)(x)v\big)=v^{-1}\o \rho(x)v^{0}+[x, v^{-1}]\o v^{0}-x_{(2)}\o \vp(x_{(1)})v,& \label{eq:yd2}\\
 &\hspace{-4mm}(\xi-\eta)(\rho(x)v)=x_{(1)}\o \rho(x_{(2)})v+v_{-1}\o \rho(x)v_{0}-v^{-1}\o \rho(x)v^{0}+x\ci v_{-1}\o v_{0}-x\ci v^{-1}\o v^{0},& \label{eq:yd3}\\
 &(\xi-\eta)(\vp(x)v)=x_{(1)}\o \vp(x_{(2)})v-x_{(2)}\o \vp(x_{(1)})v+v_{-1}\o \vp(x)v_{0}-v^{-1}\ci x\o v^{0},& \label{eq:yd4}
 \end{eqnarray}
 where $[x, y]:=x\ci y-y\ci x$.

Assume in addition that
$((V, \rho, \vp), \a)$ is a representation of $((A, \ci), N)$ and $((V, \xi, \eta), \b)$ is a corepresentation of $((A, \D), S)$ such that Eqs.(\ref{eq:dual3}), (\ref{eq:dual2}), (\ref{eq:hj}), (\ref{eq:hj2}) and the following conditions hold:
 \begin{eqnarray}
 &\a(v)_{-1}\o \b(\a(v)_{0})+N^2(v_{-1})\o v_{0}=N(v_{-1})\o \b(v_{0})+N(\a(v)_{-1})\o \a(v)_{0},& \label{eq:yd7}\\
 &\a(v)^{-1}\o \b(\a(v)^{0})+N^2(v^{-1})\o v^{0}=N(v^{-1})\o \b(v^{0})+N(\a(v)^{-1})\o \a(v)^{0},& \label{eq:yd10}\\
 &S(\a(v)^{-1})\o \a(v)^{0}+v^{-1}\o \a^2(v^{0})=S(v^{-1})\o \a(v^{0})+\a(v)^{-1}\o \a(\a(v)^{0}),& \label{eq:yd8}\\
 &S(\a(v)_{-1})\o \a(v)_{0}+v_{-1}\o \a^2(v_{0})=S(v_{-1})\o \a(v_{0})+\a(v)_{-1}\o \a(\a(v)_{0}),& \label{eq:yd9}
 \end{eqnarray}
 then we call the triple $((V, \rho, \vp, \xi, \eta), \a, \b)$ a {\bf crossed module of the Nijenhuis pre-Lie bialgebra $((A, \ci,$ $\D), N, S)$}.

 All the crossed modules of $((A, \ci, \D), N, S)$ together with the homomorphisms between the representations of $((A, \ci), N)$ and the homomorphisms between the corepresentations of $((A, \D), S)$ form a category denoted by $_{((A, \ci, \D), N, S)}^{((A, \ci, \D), N, S)}{\mathcal{CM}}$.
 \end{defi}

 The definition imposes several compatibility relations, so an effective source of examples is important. In the triangular case, the tensor defining the pre-Lie cobracket also supplies explicit coactions on a representation. The following theorem shows that these coactions satisfy the crossed-module conditions under the stated admissibility assumptions.

 \begin{thm}\label{thm:qtlrep} Let $((A, \ci), \D_r, N, S)$ be a triangular Nijenhuis pre-Lie bialgebra, $\b: V\to V$ be a linear map and $((V, \rho, \vp), \a)$ be a representation of $((A, \ci), N)$ such that $((A, \ci), N)$ is $\b$-admissible with respect to $(V, \rho, \vp)$ and Eqs.(\ref{eq:hj}), (\ref{eq:hj2}) hold. For all $v\in V$, define linear maps $\xi, \eta: V\to A\otimes V$ by
 \begin{eqnarray}
 &\xi(v):=-r^1\otimes (\rho-\vp)(r^2)v,\quad\eta(v):=r^1\otimes \vp(r^2)v.&\label{eq:cm1}
 \end{eqnarray}
 Then $((V, \rho, \vp, \xi, \eta), \a, \b)$ is a crossed module of $((A, \ci), \D_r, N, S)$.
 \end{thm}

\begin{proof}
\textbf{Step 1.} We prove that $((V,\xi,\eta),\beta)$ is a corepresentation of $((A,\Delta_r),S)$. Substituting the coactions \eqref{eq:cm1} into the left-hand side of \eqref{eq:crep1} and applying \eqref{eq:qt2-1} and \eqref{eq:rep2}, the resulting terms cancel by the representation identity \eqref{eq:rep1}. The analogous expansion using \eqref{eq:rep2} proves \eqref{eq:crep2}. For the Nijenhuis corepresentation identities, substitute \eqref{eq:cm1} and the compatibility of $r$ expressed in \eqref{eq:cybe-1}. The resulting operator identities are \eqref{eq:dual3} and \eqref{eq:dual2}; hence \eqref{eq:ncrep1} and \eqref{eq:ncrep2} hold.

\textbf{Step 2.} For the crossed-module conditions, the difference between the two sides of \eqref{eq:yd1}, after substitution of \eqref{eq:cm1}, is a sum of terms of the form $a_i\o T_i(x)v$. Here $T_i(x)$ is the difference between the representation expressions for $[\rho(x)-\varphi(x),\rho(b_i)-\varphi(b_i)]$ and the action of $[x,b_i]_\ci$. It vanishes by \eqref{eq:rep1} and \eqref{eq:rep2}. Applying the same substitutions to the other action--coaction combinations proves \eqref{eq:yd2} and \eqref{eq:yd4} by \eqref{eq:rep2}, and \eqref{eq:yd3} by \eqref{eq:rep1}.

\textbf{Step 3.} Finally, substituting \eqref{eq:cm1} and \eqref{eq:cybe-1} into the Nijenhuis crossed-module conditions reduces \eqref{eq:yd7} to \eqref{eq:hj} and \eqref{eq:hj2}. The same calculation proves \eqref{eq:yd10} from \eqref{eq:hj2}, \eqref{eq:yd8} from \eqref{eq:nrep2}, and \eqref{eq:yd9} from \eqref{eq:nrep1} and \eqref{eq:nrep2}. Thus all required identities hold.
\end{proof}

The construction above is also compatible with morphisms. This is useful because it places a class of Nijenhuis pre-Lie representations inside the crossed-module framework in a way that respects their maps, rather than producing isolated objects. To formulate this observation, let $((A, \ci), \D_r, N, S)$ be a triangular Nijenhuis pre-Lie bialgebra, and let $_{((A, \circ), \Delta_r, N, S)}\mathcal{M}$ denote the category defined as follows.

 An {\bf object} is a quadruple $\big(((V, \rho, \varphi), \alpha), \beta \big)$ such that:
 \begin{enumerate}
 \item $((V, \rho, \varphi), \alpha)$ is a representation of $((A,\circ),N)$;
 \item $((A, \circ), N)$ is $\beta$-admissible with respect to $(V, \rho, \varphi)$; and
 \item Eqs.\eqref{eq:hj} and \eqref{eq:hj2} hold.
 \end{enumerate}
 A {\bf morphism} from $\big(((V, \rho, \varphi), \alpha), \beta\big)$ to $\big(((V', \rho', \varphi'), \alpha'), \beta' \big)$ is a linear map $f\colon V\to V'$ which is a homomorphism of representations from $((V,\rho,\varphi),\alpha)$ to $((V', \rho', \varphi'), \alpha')$ of $((A, \circ), N)$ and satisfies $f\beta=\beta'f$. Then we have the following result.
 \begin{cor}\label{cor:thm:qtlrep}
  The functor
   \begin{tabular}{cccc}
          $\mathfrak{F}:$ & $_{((A,\circ),\Delta_r,N,S)}\mathcal{M}$  &$\longrightarrow$ &$_{((A, \ci, \D), N, S)}^{((A, \ci, \D), N, S)}{\mathcal{CM}}$  \\
            & $V$ & $\mapsto$  & $V$  \\
            & $f$ & $\mapsto$  & $f$ \\
        \end{tabular}
   is an embedding  functor.
 \end{cor}

 \begin{proof} According to Theorem \ref{thm:qtlrep}, we only need to verify that if $f: V\to V'$ is a homomorphism of corepresentations from  $(V, \xi, \eta)$ to $(V', \xi', \eta')$ of the pre-Lie coalgebra $(A, \D_r)$, where $\xi, \eta, \xi', \eta'$ are given by Eq.\eqref{eq:cm1}. For all $v\in V$, one has
 \begin{eqnarray*}
 &f(v)_{-1}\otimes f(v)_{0}-v_{-1}\otimes f(v_{0})
 \stackrel{(\ref{eq:cm1})}{=}-r^1\otimes (\rho'-\vp')(r^2)f(v)+r^1\otimes f((\rho-\vp)(r^2)v)=0,&\\
 &f(v)^{-1}\otimes f(v)^{0}-v^{-1}\otimes f(v^{0})
 \stackrel{(\ref{eq:cm1})}{=}r^1\otimes \vp'(r^2)f(v)-r^1\otimes f(\vp(r^2)v)=0,&
 \end{eqnarray*}
 finishing the proof.
 \end{proof}

 \begin{rmk} \noindent In duality to Theorem~\ref{thm:qtlrep}, one can also construct a Nijenhuis pre-Lie crossed module via a dual triangular Nijenhuis pre-Lie bialgebra. We omit the detailed proof for brevity.
 \end{rmk}

\subsection{Braided Nijenhuis pre-Lie bialgebras}\label{se:rbip}
We now explain why the crossed-module conditions are natural from the viewpoint of constructing larger bialgebras. The biproduct idea is to combine the product on $A$ with an action on $V$, and dually to combine the coproduct on $A$ with a coaction on $V$. Both constructions must preserve their respective Nijenhuis operators. However, having an algebra structure and a coalgebra structure on $A\oplus V$ separately does not yet ensure the bialgebra compatibility. The purpose of this subsection is to identify the additional conditions that make them fit together.

We first isolate the algebraic and coalgebraic parts of the construction. These preliminary results prepare the ground for the biproduct criterion and the resulting notion of a braided Nijenhuis pre-Lie bialgebra.

 \begin{defi} \label{de:ma2} Let $((A, \ci), N)$ and $((V, \dd), \a)$ be two Nijenhuis pre-Lie algebras and $((V, \rho, \vp), \a)$ be a representation of $((A, \ci), N)$. Then a triple $(((V, \dd), \a), \rho, \vp)$ is called a {\bf representation algebra of $((A, \ci), N)$} if for all $x\in A, u, v\in V$,
 \begin{eqnarray}
 &\rho(x)(u\dd v)=\big((\rho-\vp)(x)u\big)\dd v+u\dd (\rho(x)v),&\label{eq:ma1}\\
 &\vp(x)([u, v]_\dd)=u\dd (\vp(x)v)-v\dd (\vp(x)u),&\label{eq:ma2}
 \end{eqnarray}
 where $[u, v]_\dd:=u\dd v-v\dd u$.
 \end{defi}

 \begin{lem}\label{lem:smashp2} Let $((A, \ci), N)$ and $((V, \dd), \a)$ be two Nijenhuis pre-Lie algebras and $\rho, \vp: A\lr \End(V)$ be two linear maps. For all $x, y\in A,~u, v\in V$, define a binary operation $\ci_{\#}$ and a linear map $N+\a$ on $A\oplus V$ by
 \begin{eqnarray}
 &(x+u)\ci_{\#}(y+v):=x\ci y+\rho(x)v+\vp(y)u+u\dd v,& \label{eq:smashp2}\\
 &(N+\a)(x+u):=N(x)+\a(u).&\label{eq:smashp2-1}
 \end{eqnarray}
 Then $((A\oplus V, \ci_{\#}), N+\a)$ is a Nijenhuis pre-Lie algebra if and only if
 \begin{enumerate}[(1)]
   \item $((V, \rho, \vp), \a)$ is a representation of $((A, \ci), N)$;
   \item $(((V, \dd), \a), \rho, \vp)$ is a representation algebra of $((A, \ci), N)$.
 \end{enumerate}
 In this case, we call this Nijenhuis pre-Lie algebra a {\bf Nijenhuis pre-Lie product algebra}, denoted by $(A\# V, N+\a)$.
 \end{lem}

\begin{proof} For all $x, y, z\in A$ and $u, v, w\in V$, based on Eqs.(\ref{eq:smashp2}) and (\ref{eq:smashp2-1}), we compute
{\small\begin{eqnarray*}
 &&\hspace{-3mm}((x+u)\ci_{\#}(y+v))\ci_{\#}(z+w)-(x+u)\ci_{\#}((y+v)\ci_{\#}(z+w))\\
 &&=(x\ci y)\ci z+\rho(x\ci y)w+\vp(z)(\rho(x)v)+\vp(z)(\vp(y)u)+\vp(z)(u\dd v)+\rho(x)v\dd w+\vp(y)u\dd w+(u\dd v)\dd w\\
 &&\quad-x\ci (y\ci z)-\rho(x)(\rho(y)w)-\rho(x)(\vp(z)v)-\rho(x)(v\dd w)-\vp(y\ci z)u-u\dd \rho(y)w-u\dd \vp(z)v-u\dd (v\dd w);\\
 &&\hspace{-3mm}((y+v)\ci_{\#}(x+u))\ci_{\#}(z+w)-(y+v)\ci_{\#}((x+u)\ci_{\#}(z+w))\\
 &&=(y\ci x)\ci z+\rho(y\ci x)w+\vp(z)(\rho(y)u)+\vp(z)(\vp(x)v)+\vp(z)(v\dd u)+\rho(y)u\dd w+\vp(x)v\dd w+(v\dd u)\dd w\\
 &&\quad-y\ci (x\ci z)-\rho(y)(\rho(x)w)-\rho(y)(\vp(z)u)-\rho(y)(u\dd w)-\vp(x\ci z)v-v\dd \rho(x)w-v\dd \vp(z)u-v\dd (u\dd w);\\
 &&\hspace{-3mm}(N+\a)(x+u)\ci_{\#}(N+\a)(y+v))+(N+\a)^2((x+u)\ci_{\#}(y+v))\\
 &&=N(x)\ci N(y)+\rho(N(x))\a(v)+\vp(N(y))\a(u)+\a(u)\dd \a(v)+N^2(x\ci y)+\a^2(\rho(x)v)+\a^2(\vp(y)u)+\a^2(u\dd v);\\
 &&\hspace{-3mm}(N+\a)\big((N+\a)(x+u)\ci_{\#}(y+v)\big)+(N+\a)\big((x+u)\ci_{\#}(N+\a)(y+v)\big)\\
 &&=N(N(x)\ci y)+\a(\rho(N(x))v)+\a(\vp(y)\a(u))+\a(\a(u)\dd v)+N(x\ci N(y))+\a(\rho(x)\a(v))+\a(\vp(N(y))u)\\
 &&\quad+\a(u\dd \a(v)).
 \end{eqnarray*}
 }
 Then the proof is completed by comparing the above equations.
 \end{proof}

We next pass to the dual construction. A suitable corepresentation structure on $V$ allows the coproduct to be extended to $A\oplus V$, and the corresponding compatibility is expressed by the following definition and lemma.

 \begin{defi} \label{de:cmca2} Let $((A, \D), S)$ and $((V, \nabla), \b)$ be two Nijenhuis pre-Lie coalgebras and $((V, \xi, \eta), \b)$ be a corepresentation of $((A, \D), S)$. Then a triple $(((V, \nabla), \b), \xi, \eta)$ called a {\bf corepresentation coalgebra of $((A, \D), S)$} if for all $u\in V$,
 \begin{eqnarray}
 &u^{-1}\o (u^{0(1)}\o u^{0(2)}-u^{0(2)}\o u^{0(1)})=u^{(2)-1}\o (u^{(1)} \o u^{(2)0}-u^{(2)0}\o u^{(1)}),&\label{eq:cmca2}\\
 &u_{-1}\o {u_{0}}^{(1)}\o {u_{0}}^{(2)}=({u^{(1)}}_{-1}\o {u^{(1)}}_{0}-u^{(1)-1}\o u^{(1)0})\o u^{(2)}+{u^{(2)}}_{-1}\o u^{(1)}\o {u^{(2)}}_{0}.& \label{eq:cmca2-1}
 \end{eqnarray}
 \end{defi}

 \begin{lem}\label{lem:smashcp2} Let $((A, \D), S)$ and $((V, \nabla), \b)$ be two Nijenhuis pre-Lie coalgebras and $\xi, \eta: V\lr A\o V$ be two linear maps. For all $x\in A,~u\in V$, define a coproduct $\D_{\times}$ and a linear map $S+\b$ on $A\oplus V$ by
 \begin{eqnarray}
 &\D_{\times}(x+u):=(x+u)_{<1>}\o (x+u)_{<2>}:=x_{(1)}\o x_{(2)}+u_{-1}\o u_{0}+u^{0}\o u^{-1}+u^{(1)}\o u^{(2)},& \label{eq:smashcp2}\\
 &(S+\b)(x+u):=S(x)+\b(u).&\label{eq:smashcp2-1}
 \end{eqnarray}
 Then $((A\oplus V, \D_{\times}), S+\b)$ is a Nijenhuis pre-Lie coalgebra if and only if
 \begin{enumerate}[(1)]
   \item $((V, \xi, \eta), \b)$ is a corepresentation of $((A, \D), S)$;
   \item $(((V, \nabla), \b), \xi, \eta)$ is a corepresentation coalgebra of $((A, \D), S)$.
 \end{enumerate}
 In this case, we call this Nijenhuis pre-Lie coalgebra a {\bf Nijenhuis pre-Lie coproduct coalgebra}, denoted by $(A\times V, S+\b)$.
 \end{lem}

 \begin{proof} Dual to the proof of Lemma \ref{lem:smashp2}.
 \end{proof}

The preceding two lemmas address the Nijenhuis algebra and coalgebra structures separately. To combine them into a bialgebra, we also need conditions expressing how the action interacts with the coproduct and how the coaction interacts with the product. The next two definitions record precisely these mixed requirements.

 \begin{defi} \label{de:mca2} Let $(A, \ci)$ be a pre-Lie algebra and $(V, \nabla)$ a pre-Lie coalgebra and $(V, \rho, \vp)$ a representation of $(A, \ci)$. Then a triple $((V, \nabla), \rho, \vp)$ is called a {\bf representation coalgebra of $(A, \ci)$} if for all $x\in A, u\in V$,
 \begin{eqnarray}
 &\nabla\big((\rho-\vp)(x)u\big)=\rho(x)u^{(1)}\o u^{(2)}+u^{(1)}\o (\rho-\vp)(x)u^{(2)},&\label{eq:mca2}\\
 &(\id\o \id-\sigma)\nabla\big(\rho(x)u\big)=(\id\o \id-\sigma)\big(\rho(x)u^{(1)}\o u^{(2)}+u^{(1)}\o \rho(x)u^{(2)}\big),&\label{eq:mca2-1}\\
 &(\id\o \id-\sigma)\nabla\big(\vp(x)u\big)=(\id\o \id-\sigma)\big(u^{(1)}\o \vp(x)u^{(2)}\big).&\label{eq:mca2-2}
 \end{eqnarray}
 \end{defi}

 \begin{defi} \label{de:cma2} Let $(A, \D)$ be a pre-Lie coalgebra and $(V, \dd)$ a pre-Lie algebra and $(V, \xi, \eta)$ a corepresentation of $(A, \D)$. Then a triple $((V, \dd), \xi, \eta)$ is called a {\bf corepresentation algebra of $(A, \D)$} if for all $u, v\in V$,
 \begin{eqnarray}
 &\xi([u, v]_{\dd})=u_{-1}\o [u_{0}, v]_{\dd}+v_{-1}\o [u, v_{0}]_{\dd}&\label{eq:cma2}\\
 &\eta([u, v]_{\dd})=v^{-1}\o u\dd v^{0}-u^{-1}\o v\dd u^{0},&\label{eq:cma2-1}\\
 &(\xi-\eta)(u\dd v)=u_{-1}\o u_{0}\dd v+v_{-1}\o u\dd v_{0}-v^{-1}\o u\dd v^{0}.&\label{eq:cma2-2}
 \end{eqnarray}
 \end{defi}

With the algebraic, coalgebraic, and mixed compatibility requirements in place, we can state the main structural result of this section. It gives necessary and sufficient conditions for the direct-sum construction to be a Nijenhuis pre-Lie bialgebra, thereby explaining the role of crossed modules in the biproduct framework.

\begin{thm}\label{thm:rdnbip}
Let $((A, \ci, \D), N, S)$ be a Nijenhuis pre-Lie bialgebra.
Let $(((V, \dd), \a), \rho, \vp)$ be a representation algebra of $((A, \ci), N)$ such that $((V, \dd), \a)$ is $\b$-admissible, and $(((V, \nabla), \b), \xi, \eta)$ be a corepresentation coalgebra of $((A, \D), S)$ such that $((V^*, \nabla^*), \b^*)$ is $\a^*$-admissible.
Consider $((A\oplus V, \ci_{\#}, \D_{\times}), N+\a, S+\b)$ consisting of the Nijenhuis pre-Lie product algebra $(A\# V, N+\a)$ defined in Lemma~\ref{lem:smashp2}
and the Nijenhuis pre-Lie coproduct coalgebra $(A\times V, S+\b)$  defined in Lemma~\ref{lem:smashcp2}.
Then  $((A\oplus V, \ci_{\#}, \D_{\times}), N+\a, S+\b)$ is a Nijenhuis pre-Lie bialgebra if and only if
 \begin{enumerate}[(1)]
   \item \label{it:thm:rdnbip1} $((V, \rho, \vp, \xi, \eta), \a, \b)$ is a crossed module of $((A, \ci, \D), N, S)$;
   \item \label{it:thm:rdnbip2} $((V, \nabla), \rho, \vp)$ is a representation coalgebra of $(A, \ci)$;
   \item \label{it:thm:rdnbip3} $((V, \dd), \xi, \eta)$ is a corepresentation algebra of $(A, \D)$;
   \item \label{it:thm:rdnbip4} for all $u, v\in V$,
   \begin{eqnarray}
   &&\hspace{-11mm}\nabla([u,v]_{\dd})=u\dd v^{(1)}\o v^{(2)}+v^{(1)}\o [u,v^{(2)}]_{\dd}-v\dd u^{(1)}\o u^{(2)}-u^{(1)}\o [v,u^{(2)}]_{\dd}+\psi_1(u\o v)\label{eq:yd11}\\
   &&\hspace{-11mm} (\id\o \id-\sigma)\nabla(u\dd v)=(\id\o \id-\sigma)(u^{(1)}\o u^{(2)}\dd v+v^{(1)}\o u\dd v^{(2)}\nonumber\\
   &&\hspace{84mm}-v^{(2)}\o u\dd v^{(1)})+\psi_2(u\o v),\label{eq:yd12}
   \end{eqnarray}
 where
 \begin{eqnarray}
 &&\psi_1(u\o v)=\vp(v_{-1})u\o v_{0}-v^{0}\o (\rho-\vp)(v^{-1})u-\vp(u_{-1})v\o u_{0}+u^{0}\o (\rho-\vp)(u^{-1})v,\label{eq:yd11-1}\\
 &&\psi_2(u\o v)=(\id\o \id-\sigma)(u^{0}\o \rho(u^{-1})v+v^{0}\o \vp(v^{-1})u-v_{0}\o \vp(v_{-1})u).\label{eq:yd12-1}
 \end{eqnarray}
 \end{enumerate}
In this case, we call this bialgebra a {\bf Nijenhuis pre-Lie biproduct bialgebra}, denoted by $(A^{\#}_{\times} V, N+\a, S+\b)$.
 \end{thm}

\begin{proof} For all $x, y\in A$ and $u, v\in V$, we compute
 \begin{eqnarray*}
 &&\hspace{-6mm}\D_{\times}((x+u)\ci_{\#} (y+v)-(y+v)\ci_{\#} (x+u))\\
 &&=(x\ci y)_{(1)}\o (x\ci y)_{(2)}+(\rho(x)v)_{-1}\o (\rho(x)v)_{0}+(\rho(x)v)^{0}\o (\rho(x)v)^{-1}+(\rho(x)v)^{(1)}\o (\rho(x)v)^{(2)}\\
 &&\quad +(\vp(y)u)_{-1}\o (\vp(y)u)_{0}+(\vp(y)u)^{0}\o (\vp(y)u)^{-1}+(\vp(y)u)^{(1)}\o (\vp(y)u)^{(2)}+(u\dd v)_{-1}\o (u\dd v)_{0}\\
 &&\quad +(u\dd v)^{0}\o (u\dd v)^{-1}+(u\dd v)^{(1)}\o (u\dd v)^{(2)}-(y\ci x)_{(1)}\o (y\ci x)_{(2)}-(\rho(y)u)_{-1}\o (\rho(y)u)_{0}\\
 &&\quad-(\rho(y)u)^{0}\o (\rho(y)u)^{-1}-(\rho(y)u)^{(1)}\o (\rho(y)u)^{(2)}-(\vp(x)v)_{-1}\o (\vp(x)v)_{0}-(\vp(x)v)^{0}\o (\vp(x)v)^{-1}\\
 &&\quad-(\vp(x)v)^{(1)}\o (\vp(x)v)^{(2)}-(v\dd u)_{-1}\o (v\dd u)_{0}-(v\dd u)^{0}\o (v\dd u)^{-1}-(v\dd u)^{(1)}\o (v\dd u)^{(2)},
 \end{eqnarray*}
\begin{eqnarray*}
 &&\hspace{-6mm}(x+u)\ci_{\#} (y+v)_{<1>}\o (y+v)_{<2>}+(y+v)_{<1>}\o (x+u)\ci_{\#} (y+v)_{<2>}\\
 &&-(y+v)_{<1>}\o (y+v)_{<2>}\ci_{\#} (x+u)-(y+v)\ci_{\#} (x+u)_{<1>}\o (x+u)_{<2>}\\
 &&\quad-(x+u)_{<1>}\o (y+v)\ci_{\#} (x+u)_{<2>}+(x+u)_{<1>}\o (x+u)_{<2>}\ci_{\#} (y+v)\\
 &&\hspace{-3mm}=x\ci y_{(1)}\o y_{(2)}+\vp(y_{(1)})u\o y_{(2)}+x\ci v_{-1}\o v_{0} +\vp(v_{-1})u\o v_{0}+\rho(x)v^{0}\o v^{-1}+u\dd v^{0}\o v^{-1}\\
 &&+\rho(x)v^{(1)}\o v^{(2)}+u\dd v^{(1)}\o v^{(2)}+y_{(1)}\o x\ci y_{(2)}+y_{(1)}\o \vp(y_{(2)})u+v_{-1}\o \rho(x)v_{0}+v_{-1}\o u\dd v_{0}\\
 &&+v^{0}\o x\ci v^{-1}+v^{0}\o \vp(v^{-1})u+v^{(1)}\o \rho(x)v^{(2)}+v^{(1)}\o u\dd v^{(2)}-y_{(1)}\o \rho(y_{(2)})u-y_{(1)}\o y_{(2)}\ci x\\
 &&-v_{-1}\o \vp(x)v_{0}-v_{-1}\o v_{0}\dd u-v^{0}\o v^{-1}\ci x-v^{0}\o \rho(v^{-1})u-v^{(1)}\o \vp(x)v^{(2)}-v^{(1)}\o v^{(2)}\dd u\\
 &&-y\ci x_{(1)}\o x_{(2)}-\vp(x_{(1)})v\o x_{(2)}-y\ci u_{-1}\o u_{0}-\vp(u_{-1})v\o u_{0}-\rho(y)u^{0}\o u^{-1}-v\dd u^{0}\o u^{-1}\\
 &&-\rho(y)u^{(1)}\o u^{(2)}-v\dd u^{(1)}\o u^{(2)}-x_{(1)}\o y\ci x_{(2)}-x_{(1)}\o \vp(x_{(2)})v-u_{-1}\o \rho(y)u_{0}-u_{-1}\o v\dd u_{0}\\
 &&-u^{0}\o y\ci u^{-1}-u^{0}\o \vp(u^{-1})v-u^{(1)}\o \rho(y)u^{(2)}-u^{(1)}\o v\dd u^{(2)}+x_{(1)}\o \rho(x_{(2)})v+x_{(1)}\o x_{(2)}\ci y\\
 &&+u_{-1}\o \vp(y)u_{0}+u_{-1}\o u_{0}\dd v+u^{0}\o u^{-1}\ci y+u^{0}\o \rho(u^{-1})v+u^{(1)}\o \vp(y)u^{(2)}+u^{(1)}\o u^{(2)}\dd v,
 \end{eqnarray*}
\begin{eqnarray*}
 &&\hspace{-6mm}(\id\o \id-\sigma)\D_{\times}((x+u)\ci_{\#} (y+v))\\
 &&=(x\ci y)_{(1)}\o (x\ci y)_{(2)}+(\rho(x)v)_{-1}\o (\rho(x)v)_{0}+(\rho(x)v)^{0}\o (\rho(x)v)^{-1}+(\rho(x)v)^{(1)}\o (\rho(x)v)^{(2)}\\
 &&\quad +(\vp(y)u)_{-1}\o (\vp(y)u)_{0}+(\vp(y)u)^{0}\o (\vp(y)u)^{-1}+(\vp(y)u)^{(1)}\o (\vp(y)u)^{(2)}+(u\dd v)_{-1}\o (u\dd v)_{0}\\
 &&\quad +(u\dd v)^{0}\o (u\dd v)^{-1}+(u\dd v)^{(1)}\o (u\dd v)^{(2)}
 -(x\ci y)_{(2)}\o (x\ci y)_{(1)}-(\rho(x)v)_{0}\o (\rho(x)v)_{-1}\\
 &&\quad-(\rho(x)v)^{-1}\o (\rho(x)v)^{0}-(\rho(x)v)^{(2)}\o (\rho(x)v)^{(1)}-(\vp(y)u)_{0}\o (\vp(y)u)_{-1}-(\vp(y)u)^{-1}\o (\vp(y)u)^{0}\\
 &&\quad -(\vp(y)u)^{(2)}\o (\vp(y)u)^{(1)}-(u\dd v)_{0}\o (u\dd v)_{-1}-(u\dd v)^{-1}\o (u\dd v)^{0}-(u\dd v)^{(2)}\o (u\dd v)^{(1)},
 \end{eqnarray*}
\begin{eqnarray*}
 &&\hspace{-6mm}(\id\o \id-\sigma)\big((x+u)_{<1>}\o (x+u)_{<2>}\ci_{\#} (y+v)+(y+v)_{<1>}\o (x+u)\ci_{\#} (y+v)_{<2>}\\
 &&\quad-(y+v)_{<2>}\o (x+u)\ci_{\#} (y+v)_{<1>}\big)\\
 &&\hspace{-3mm}=(\id\o \id-\sigma)\big(x_{(1)}\o x_{(2)}\ci y+x_{(1)}\o \rho(x_{(2)})v+u_{-1}\o \vp(y)u_{0}+u_{-1}\o u_{0}\dd v+u^{0}\o u^{-1}\ci y\\
 &&+u^{0}\o \rho(u^{-1})v+u^{(1)}\o \vp(y)u^{(2)}+u^{(1)}\o u^{(2)}\dd v+y_{(1)}\o x\ci y_{(2)}+y_{(1)}\o \vp(y_{(2)})u+v_{-1}\o \rho(x)v_{0} \\
 &&+v_{-1}\o u\dd v_{0}+v^{0}\o x\ci v^{-1}+v^{0}\o \vp(v^{-1})u+v^{(1)}\o \rho(x)v^{(2)}+v^{(1)}\o u\dd v^{(2)}-y_{(2)}\o x\ci y_{(1)}\\
 &&-y_{(2)}\o \vp(y_{(1)})u-v_{0}\o x\ci v_{-1}-v_{0}\o \vp(v_{-1})u-v^{-1}\o \rho(x)v^{0}-v^{-1}\o u\dd v^{0}-v^{(2)}\o \rho(x)v^{(1)}\\
 &&-v^{(2)}\o u\dd v^{(1)}\big).
 \end{eqnarray*}

 Thus, Eq.~\eqref{eq:b1} for $\Delta_\times, \circ_\#$ is equivalent to Eqs.~\eqref{eq:yd1}, \eqref{eq:yd2}, \eqref{eq:mca2}, \eqref{eq:cma2}, \eqref{eq:cma2-1}, \eqref{eq:yd11}, while Eq.~\eqref{eq:b2} for $\Delta_\times, \circ_\#$ is equivalent to Eqs.~\eqref{eq:yd3}, \eqref{eq:yd4}, \eqref{eq:cma2-2}, \eqref{eq:mca2-1}, \eqref{eq:mca2-2}, \eqref{eq:yd12}. Similarly, for $\circ_\#, N+\alpha, S+\beta$, Eq.~\eqref{eq:dual3-1} holds if and only if Eqs.~\eqref{eq:dual3} and \eqref{eq:hj2} hold, and Eq.~\eqref{eq:dual2-1} holds if and only if Eqs.~\eqref{eq:dual2} and \eqref{eq:hj} hold. For $\Delta_\times, N+\alpha, S+\beta$, Eq.~\eqref{eq:nliebialg1} is equivalent to Eqs.~\eqref{eq:yd7} and \eqref{eq:yd8}, while Eq.~\eqref{eq:nliebialg2} is equivalent to Eqs.~\eqref{eq:yd9} and \eqref{eq:yd10}. Therefore, by Lemmas~\ref{lem:smashp2} and \ref{lem:smashcp2}, the proof is complete.
 \end{proof}

The criterion contains correction terms that measure the interaction between the product and coproduct on $V$ and the ambient action--coaction data. For a triangular bialgebra, the specific coactions constructed earlier simplify these terms. This gives a more direct relationship between the biproduct criterion and an ordinary pre-Lie bialgebra structure on $V$.

\begin{lem}\label{lem:thm:rdnbip}
Let $(A,\ci,r,\D_r)$ be a triangular pre-Lie bialgebra, let
$(V,\rho,\vp)$ be a representation of $(A,\ci)$, and define
$\xi,\eta:V\to A\otimes V$ by Eq.~\eqref{eq:cm1}.  Then the correction
tensors $\psi_1$ and $\psi_2$ from Eqs.~\eqref{eq:yd11-1}
and~\eqref{eq:yd12-1} vanish identically.
Consequently, if $(V,\dd)$ is a pre-Lie algebra and
$(V,\nabla)$ is a pre-Lie coalgebra, then the two equalities
\eqref{eq:yd11} and \eqref{eq:yd12} are equivalent to the
bialgebra compatibility identities for $(V,\dd,\nabla)$.
In particular, under the other hypotheses of Theorem~\ref{thm:rdnbip},
Item~\ref{it:thm:rdnbip4} of that theorem holds if and only if
$(V,\dd,\nabla)$ is a pre-Lie bialgebra.
\end{lem}
\begin{proof}
Write $r=\sum_i a_i\otimes b_i$ and $D=\rho-\vp$.
By Eq.~\eqref{eq:cm1}, the two coactions, in the distinct
Sweedler conventions of Section~6, are
$$
  \xi(w)=w_{-1}\otimes w_0
  =-\sum_i a_i\otimes D(b_i)w,\qquad
  \eta(w)=w^{-1}\otimes w^0
  =\sum_i a_i\otimes\vp(b_i)w.
$$
Symmetry of $r$ means that $\sum_i a_i\otimes b_i
=\sum_i b_i\otimes a_i$, and therefore, for any maps
$F,G:A\to\operatorname{End}(V)$,
$$
 \sum_i F(a_i)u\otimes G(b_i)v
 =\sum_i F(b_i)u\otimes G(a_i)v.
$$
Substituting the coactions into Eq.~\eqref{eq:yd11-1}, we obtain
\begin{align*}
 \psi_1(u\otimes v)
 =\sum_i\bigl(&-\vp(a_i)u\otimes D(b_i)v
              -\vp(b_i)v\otimes D(a_i)u\\
             &+\vp(a_i)v\otimes D(b_i)u
              +\vp(b_i)u\otimes D(a_i)v\bigr)=0.
\end{align*}
Indeed, the first and fourth summands cancel by the symmetry of $r$,
as do the second and third summands.

Similarly, Eq.~\eqref{eq:yd12-1} yields
\begin{align*}
 \psi_2(u\otimes v)
  &=(\id\otimes\id-\sigma)\sum_i\bigl(
       \vp(b_i)u\otimes\rho(a_i)v
       +\vp(b_i)v\otimes\vp(a_i)u
       +D(b_i)v\otimes\vp(a_i)u\bigr)\\
  &=(\id\otimes\id-\sigma)\sum_i\bigl(
       \vp(b_i)u\otimes\rho(a_i)v
       +\rho(b_i)v\otimes\vp(a_i)u\bigr),
\end{align*}
where $D+\vp=\rho$.
Set $W=\sum_i\vp(b_i)u\otimes\rho(a_i)v$.
Because $r$ is symmetric,
$$
 \sigma W=\sum_i\rho(a_i)v\otimes\vp(b_i)u
          =\sum_i\rho(b_i)v\otimes\vp(a_i)u.
$$
Thus $\psi_2(u\otimes v)=(\id\otimes\id-\sigma)(W+\sigma W)=0$.

Finally, with these two corrections equal to zero,
Eqs.~\eqref{eq:yd11} and~\eqref{eq:yd12} coincide word for word
with the two pre-Lie bialgebra compatibility identities
\eqref{eq:b1} and~\eqref{eq:b2}, upon replacing
$(A,\ci,\D)$ by $(V,\dd,\nabla)$.
As the pre-Lie algebra and coalgebra axioms are assumed separately,
this gives the asserted equivalence.  Note that only the symmetry of
$r$ was used to prove $\psi_1=\psi_2=0$.
\end{proof}

Motivated by Majid's braided approach to Lie bialgebras \cite{Maj}, we now use the preceding compatibility criterion to introduce the corresponding notion in the Nijenhuis pre-Lie setting.

 \begin{defi}\label{defi:bnpre} Under the assumption of Theorem \ref{thm:rdnbip}, a triple $(((V, \dd, \nabla), \rho, \vp, \xi, \eta), \a, \b)$ is called a {\bf braided Nijenhuis pre-Lie bialgebra} if Items \ref{it:thm:rdnbip1}-\ref{it:thm:rdnbip4} in Theorem \ref{thm:rdnbip} hold.
 \end{defi}

 The following corollary expresses the intended significance of the definition: the braided compatibility conditions are precisely those needed for the associated Nijenhuis pre-Lie biproduct.

 \begin{cor}\label{cor:defi:bnpre} Under the assumption of Theorem \ref{thm:rdnbip}, $(((V, \dd, \nabla), \rho, \vp, \xi, \eta), \a, \b)$ is a braided Nijenhuis pre-Lie bialgebra if and only if $(A^{\#}_{\times} V, N+\a, S+\b)$ is a Nijenhuis pre-Lie biproduct bialgebra.
 \end{cor}

The triangular setting provides a particularly concrete application of this characterization. In this case the coactions are determined by the triangular tensor, and the remaining condition can be expressed in terms of the representation coalgebra structure.

 \begin{pro}\label{pro:thm:qtlrep} Let $((A, \ci), \D_r, N, S)$ be a triangular Nijenhuis pre-Lie bialgebra and $((V, \dd, \nabla),$ $\a, \b)$ be a Nijenhuis pre-Lie bialgebra, $((V, \rho, \vp), \a)$ be a representation of $((A, \ci), N)$ and further, $(((V, \dd), \a), \rho, \vp)$ be a representation algebra of $((A, \ci), N)$. Assume also that $((A,\ci),N)$ is $\beta$-admissible with respect to $(V,\rho,\vp)$ and that Eqs.~\eqref{eq:hj} and \eqref{eq:hj2} hold. Then $((A\oplus V, \ci_{\#}, \D^r_{\times}), N+\a, S+\b)$ is a Nijenhuis pre-Lie bialgebra, where
 \begin{eqnarray*}
 &\D^r_{\times}(x+u):=x\ci a_i\o b_i+a_i\o [x, b_i]_\ci-a_i\o (\rho-\vp)(b_i)u+\vp(b_i)u\o a_i+u^{(1)}\o u^{(2)},&
 \end{eqnarray*}
 if and only if $((V, \nabla), \rho, \vp)$ is a representation coalgebra of $(A, \ci)$.
 \end{pro}

\begin{proof} The added admissibility and mixed conditions make Theorem~\ref{thm:qtlrep} applicable. By Theorems~\ref{thm:qtlrep} and~\ref{thm:rdnbip}, together with Lemma~\ref{lem:thm:rdnbip}, it remains to verify Eqs.~\eqref{eq:cmca2}, \eqref{eq:cmca2-1} and \eqref{eq:cma2}--\eqref{eq:cma2-2}. For arbitrary $u, v \in V$, we compute
{\small\begin{eqnarray*}
&& u^{-1} \otimes \big( u^{0(1)} \otimes u^{0(2)} - u^{0(2)} \otimes u^{0(1)} \big)
  -u^{(2)-1} \otimes \big( u^{(1)} \otimes u^{(2)0} -u^{(2)0} \otimes u^{(1)} \big) \\
&&\quad \stackrel{(\ref{eq:cm1})}{=}
a_i \otimes \Big( (\varphi(b_i)u)^{(1)} \otimes (\varphi(b_i)u)^{(2)}
-(\varphi(b_i)u)^{(2)} \otimes (\varphi(b_i)u)^{(1)}
 - u^{(1)} \otimes \varphi(b_i)u^{(2)} + \varphi(b_i)u^{(2)} \otimes u^{(1)} \Big)\\
 &&\hspace{3mm}\stackrel{(\ref{eq:mca2-2})}{=} 0.
\end{eqnarray*}}
Hence Eq.~\eqref{eq:cmca2} holds. Similarly, Eq.~\eqref{eq:cmca2-1} follows from Eqs.~\eqref{eq:cm1} and \eqref{eq:mca2}; Eq.~\eqref{eq:cma2} follows from Eqs.~\eqref{eq:cm1}, \eqref{eq:ma1} and \eqref{eq:ma2}; Eq.~\eqref{eq:cma2-1} follows from Eqs.~\eqref{eq:cm1} and \eqref{eq:ma2}; and Eq.~\eqref{eq:cma2-2} follows from Eqs.~\eqref{eq:cm1} and \eqref{eq:ma1}. This completes the proof.
\end{proof}

The constructions in this section show how crossed-module compatibility controls the passage from a Nijenhuis pre-Lie bialgebra and auxiliary algebra--coalgebra data to a biproduct bialgebra. In particular, the triangular construction supplies crossed modules directly from representations, while the biproduct criterion characterizes when the resulting combined structure satisfies all the required identities. This provides a structural complement to the operator constructions and Lie-theoretic results developed in the earlier sections.

\bigskip
\noindent{\bf Acknowledgments.}
This work is supported by National Natural Science Foundation of China (No. 12471033).

\smallskip

\noindent
{\bf Declaration of interests.} The authors have no conflicts of interest to disclose.

\noindent
{\bf Data availability.} Data sharing is not applicable as no data were created or analyzed.

\noindent
{\bf AI statement.} AI was used to assist the authors in correcting grammar and
spelling.

\end{document}